\documentclass{article}
\usepackage[utf8]{inputenc}
\usepackage{amssymb}
\usepackage[english]{babel}
\usepackage{amssymb}
\usepackage{amsmath}
\usepackage{xcolor}
\usepackage[numbers]{natbib}
\usepackage{bbm}
\usepackage{graphicx}
\usepackage{hyperref}
\usepackage{amsthm}
\usepackage{comment}
\usepackage{amsfonts}
\usepackage{centernot}
\usepackage{mathtools}
\usepackage{stmaryrd}
\usepackage{fullpage}
\usepackage{subfig}
\usepackage{tabularx}
\newtheorem{theorem}{Theorem}[section]

\newtheorem{proposition}[theorem]{Proposition}

\newtheorem{remark}[theorem]{Remark}
\newtheorem{notation}[theorem]{Notation}
\newcommand{\R}{\mathbb{R}}

 \date{}
\title{Multi-species viscous models for tissue growth: incompressible limit and qualitative behaviour}
\author{
  Pierre Degond \footnote{Université de Toulouse, CNRS, Institut de Mathématiques de Toulouse, UMR 5219, F–31062 Toulouse Cedex 9, France. Email addresses: pierre.degond@math.univ-toulouse.fr, michele.romanos@math.univ-toulouse.fr, ariane.trescases@math.univ-toulouse.fr}
  \and
  Sophie Hecht \footnote{Sorbonne Université, CNRS,  Université Paris Cité, Laboratoire Jacques-Louis Lions, UMR 7598, F-75005 Paris, France. Email address: sophie.hecht@sorbonne-universite.fr}
  \and
  Michèle Romanos \footnotemark[1] 
  \and
  Ariane Trescases \footnotemark[1]
}

\begin{document}
\maketitle

\abstract{We introduce two 2D mechanical models reproducing the evolution of two viscous tissues in contact. Their main property is to model the swirling cell motions while keeping the tissues segregated, as observed during vertebrate embryo elongation. Segregation is encoded differently in the two models: by passive or active segregation (based on a mechanical repulsion pressure). We compute the incompressible limits of the two models, and obtain strictly segregated solutions. The two models thus obtained are compared. A striking feature in the active segregation model is the persistence of the repulsion pressure at the limit: a \emph{ghost effect} is discussed and confronted to the biological data. Thanks to a transmission problem formulation at the incompressible limit, we show a pressure jump at the tissues' boundaries.}\\

\tableofcontents


\section{Introduction}\label{intro}

During morphogenesis, biological shapes emerge as a result of cell and tissue dynamic interactions. For instance, during vertebrate morphogenesis the embryonic body shape is extending along the head-to-tail axis. Posterior tissues are organized in a very specific manner: the axially located neural tube (NT), which will form the future spinal cord, is surrounded by two stripes of presomitic mesoderm (PSM), which will give rise to muscle and vertebrae (Figure~\ref{sketch}\protect \subref{emb}). Thanks to live imaging and microscopy techniques, the cell and tissue dynamics of vertebrate development have been precisely described \citep{B2017, eli}. Two striking phenomena are observed during embryo development.

The first one is the segregation of the NT and the PSM cells all along the axis, with the exception of a localized mixing area close to the progenitor zone (PZ). Cells from the PZ (depicted in yellow in Figure~\ref{sketch}\protect \subref{emb}) migrate into the NT and the PSM, thus contributing to the tissues elongation. The segregation between the NT and the PSM is maintained throughout embryonic growth.

The second interesting observation is the appearance of swirling motions or vortices within the PSM, due to cell movements.  Distinct medio-lateral cell movements are seen in the posterior zone of the embryo corresponding to the exit of progenitor cells into the PSM and the NT. Anteriorly, cell movements exhibit more complicated patterns. Globally we see latero-medial cell vortices forming, but a closer look at the anterior region reveals some ajacent areas displaying medio-lateral and latero-medial vortices \citep{B2017}. The cell rotational movements are schematically represented in Figure \ref{sketch}\protect \subref{sketchrot}.

Based on these observations, our aim is to build and study mathematical models that allow to recover and explain these two phenomena in growing tissues.

\begin{figure}[h]
\centering
\begin{tabular}{cc}
\subfloat[Quail embryo, brightfield image \citep{duval} (left) and sketches of the posterior embryo (right). ]{\includegraphics[scale=0.35]{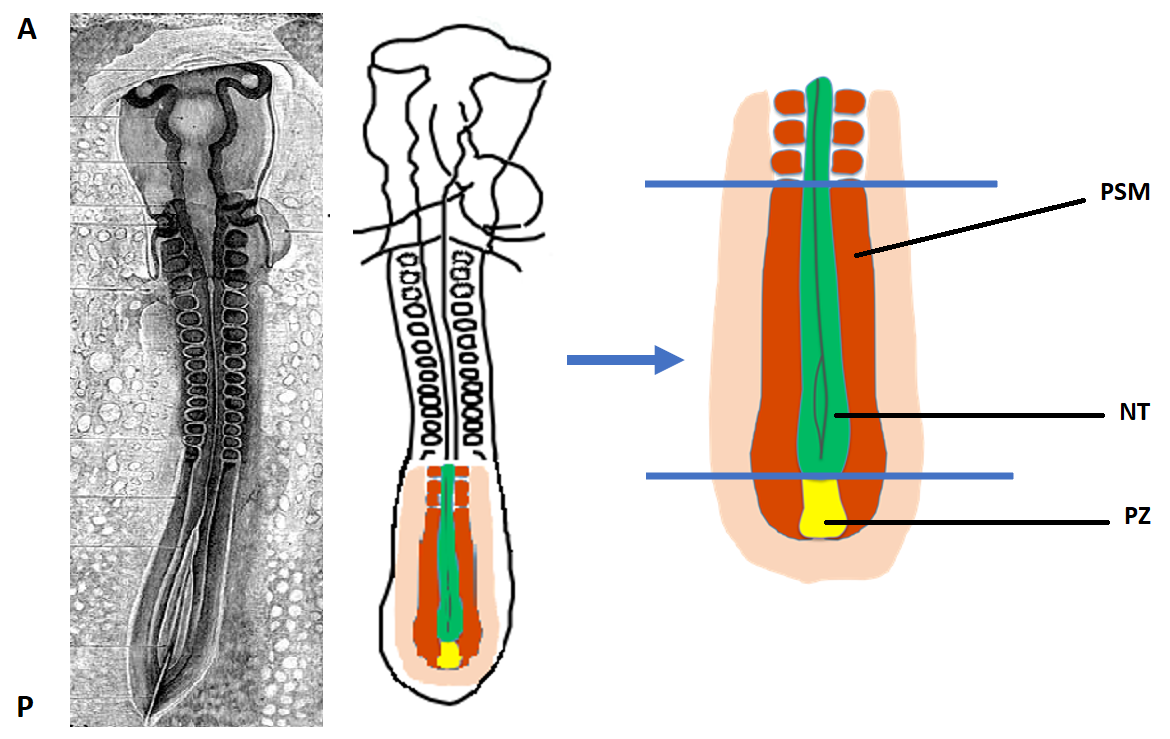}\label{emb}} 
   & \subfloat[Sketch of the cell rotations inside the PSM. As anterior cell movements are less obvious compared to the posterior ones, the arrows representing the global vortices are dashed. Small vortices of opposite directions are also observed in the anterior region.]{\includegraphics[scale=0.4]{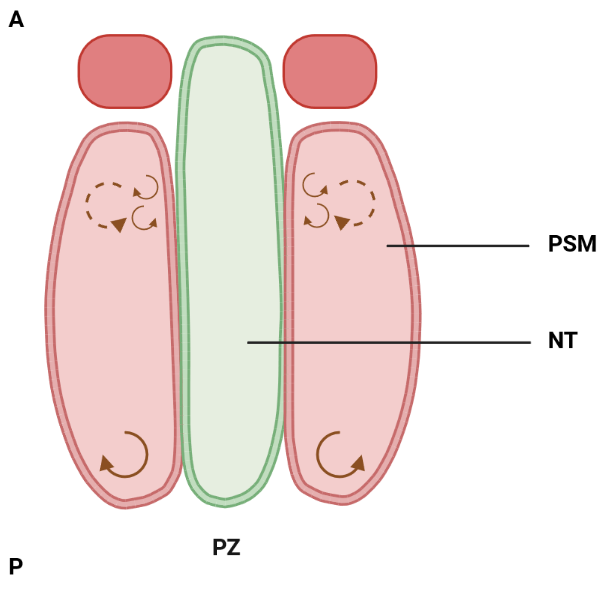}\label{sketchrot}}
   \end{tabular}
\caption{Schematics of the tissues \protect \subref{emb} and the cell movements \protect \subref{sketchrot}. The notations A and P respectively denote the anterior and the posterior parts of the embryo.}\label{sketch}
\end{figure}

In mathematical modelling, when the scale of interest is of the same order as that of a cell (or an individual), which occurs when we are interested in cell trajectories and cell-cell interactions, we turn to microscopic or agent-based models.  When the behavior of interest occurs on a tissue or collective scale, macroscopic models based on PDEs (partial differential equations) are used to describe the dynamics. They describe the evolution of some quantities related to the tissue such as density, velocity, pressure, etc.  Macroscopic models are greatly useful in the context of tissue growth \citep{tis1,ppa,tis2}. They provide a better understanding of biological phenomena  such as cell diffusion and proliferation \citep{kpp}, local and nonlocal cell interactions \citep{lnl}, chemotaxis \citep{ks}, and tumor dynamics \citep{canc7, canc2, canc3, canc4, canc5, canc1}.
In our framework, as we are interested in tissue interplay, we turn to macroscopic models to describe the embryonic growth.

To take into account the observations in the growing embryo, two modelling questions arise. Firstly, how can we model the swirling motions and their origin? And secondly, what kind of mathematical constraints should the tissues's densities obey to segregate and remain segregated? To answer the first question, we consider the tissues as viscous fluids where swirling motions originate from a nontrivial curl of the tissue velocity. Then, our approach consists of introducing tissue friction through viscosity by imposing the Brinkman law for the tissue's velocities. The Brinkman law can be derived from the Navier Stokes equations, which consists of a second order elliptic equation for the tissue velocity $v$. It takes the following form,
\begin{equation}
-\beta \Delta v +v = - \nabla p, \label{inibrink}
\end{equation}
where $p$ is the pressure inside the tissue and $\beta >0$ is the viscosity coefficient of the tissue.
When considered in a 2D (or 3D) bounded domain with Dirichlet boundary conditions, this equation produces a non-zero curl in general.

To tackle the segregation problem, it is first interesting to note that several models in literature exhibit a propagated or passive segregation \citep{B2, seg1, DHV, B1}, that is, if initially the tissues are segregated, they remain segregated for all times. On the other hand, active segregation has also been considered, for example with a chemical interplay leading to tissular separation \cite{chem}, or with a mechanical force promoting segregation \citep{PDS}.
To test the differences between a passive and an active form of segregation and their effects on tissue dynamics, we build two multi-species macroscopic PDE models describing the evolution of the tissue density, as well as its pressure and its velocity. The first model is endowed with passive segregation (VM) and the second with active segregation (ESVM) where we introduce a mechanical pressure denoted by $q$ that will enforce tissue segregation. As we would like to see swirling motions emerge, we consider the Brinkman law for the velocity in the two models.

In literature from Biology and Population dynamics, multi-species models were developed to account for interactions between different populations which affect the dynamics of the system. Early models such as the Lotka-Volterra system \citep{LV1, LV2} were studied and applied to competing/cooperating species. Other models were developed to account for attraction/repulsion between the species and volume-filling constraint, based on non-local \citep{burg1, burg2} or local effects \citep{seg1,seg2}. Emerging segregation has been studied for such models, see \citep{B2,seg1,DHV, CH, B1}.

Segregation between tissues suggests the use of a geometric description of each tissue, where we no longer look at the evolution of the density but rather at that of the tissue shape. Based on the ESVM and the VM, can we derive such a geometric description ? Indeed, geometric models are often obtained as the outcome of fluid-like evolution models by computing an asymptotic limit known as the incompressible limit.

In this paper, we compute the incompressible limit of the ESVM and of the VM and obtain two incompressible models, respectively L-ESVM and L-VM. The aim of this paper is to study these four models and investigate the links between them. 
Namely, we investigate the effects of segregation on tissue dynamics before the incompressible limit and at the incompressible limit. The two ESVM and VM models coincide when the densities are taken initially segregated (that is, active and passive segregations produce the same effect), under some additional but natural assumptions on the parameters. This equivalence between passive and active segregation for initially segregated densities holds at the incompressible limit. We also study comparatively the qualitative behavior of the two limit models.
In fact, we show that the L-ESVM and the L-VM exhibit fully segregated solutions at the limit, and yet, in the L-ESVM the active segregation force still produces a finite effect on the tissue dynamics, which we call a \emph{ghost effect}. These results are summarized in Figure \ref{fig:link}.

\begin{figure}[h!]
\centering
\includegraphics[scale=0.5]{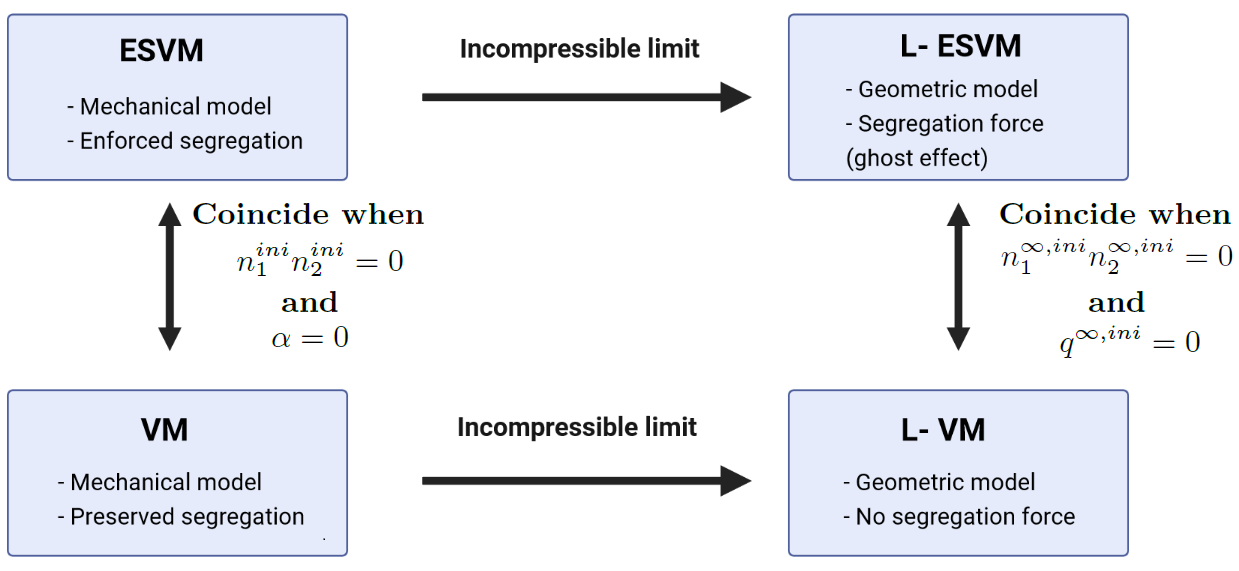}
\caption{Links between the models: ESVM, VM, L-ESVM, and L-VM.}
\label{fig:link}
\end{figure}

More precisely, the incompressible limit consists of considering the evolution model in an asymptotic regime where the pressure becomes (asymptotically) stiff. In the limit, one obtains a moving domain, which corresponds to the domain occupied by the tissue. This description of the tissue amounts to a free boundary problem. It was previously used and proved very useful in the context of tumor growth \citep{fbcan, fb2, PQF}. A free boundary approach was also used to model the elongation of the PSM \citep{bio}. It was shown that free boundary problems can be derived from evolution problems via the incompressible limit in many instances such as the porous media equation in the presence of a reaction term \citep{fb, CF, GQ, pme, fb2,PQF}, with active motion, or nutrients in the case of tumor growth \citep{NP, NS,fb2, PQF, TW}, and for the Navier-Stokes equations with growth terms \citep{EV}.

On multi-species models, the incompressible limit was done in cases which included the Darcy law for the velocity $v$, that is,
\begin{equation}
v=-\nabla p, \label{darcy}
\end{equation}
where $p$ is the pressure of the tissue, and the Brinkman law with the assumption that the velocity is the gradient of some potential $K$, that is,
\begin{equation}\label{inibrinkgrad}
v=-\nabla K \quad \text{ and } \quad -\beta \Delta K + K = p,
\end{equation}
with $\beta$ the viscosity of the tissue \citep{PDS,anyd,Debiec:2020aa,DHV, ol}. Note that when the viscosity $\beta$ is null in \eqref{inibrink}, we recover Darcy’s law \eqref{darcy}. The novelty of our work is computing the incompressible limit with a velocity following the Brinkman law with Dirichlet boundary conditions, so that in general it is not a gradient unlike in~\eqref{inibrinkgrad}.
In fact, in the Darcy case \eqref{darcy} and in the Brinkman case with gradient assumption as in \eqref{inibrinkgrad}, since the velocity takes the form of a gradient, the curl is null, which is not consistent with our observations on the vertebrate embryo.

The modelling setting of the problem is presented in the next section, together with some numerical illustrations. Section \ref{sec:mainres} is dedicated to the presentation of our main results. In Section \ref{sec:proofincomp} we compute the incompressible limit of the ESVM and derive a free boundary problem for the two tissues. The computation of the limit is followed by a discussion on the persistence of the repulsion pressure at the limit (\emph{ghost effect}). In Section \ref{sec:tran_mod1} we formulate the stationary free boundary problem as a transmission problem and prove well-posedness and regularity results on the velocities. In Section \ref{sec:tran_pres} we show the segregation property of the VM. We then prove the existence of a pressure jump across the interfaces. Finally, the results are discussed in Section~\ref{sec:discussion} and the \hyperref[sec:appendix]{Appendix} is devoted to the derivation of the transmission problem.

\section{The setting of the problem}

\subsection{Biophysical properties of the tissues} \label{biop}

The quantification of multi-tissue kinetics in \citep{B2017} demonstrates clockwise and counter clockwise vortices within each tissue along the antero-posterior axis. In our model, both tissues (PSM and NT) are endowed with a proliferation rate which we model with a pressure-dependent growth function. This choice of growth function is common in literature, as cells tend to decrease their division rate whenever they are in a high pressured environment. Furthermore,
we use a density-dependent pressure law which accounts for increased pressure in highly dense environments. We use a singular pressure law,
\begin{equation} \label{singpress} p(n)=\epsilon \,\frac{n}{1-n},
\end{equation}
with $n$ the total density for the cell populations, and $\epsilon$ a parameter. Commonly in literature, a power pressure law is used of the form: $ p(n)=n^\gamma$. Here we choose the singular pressure law \eqref{singpress} which prevents cells from overlapping. A pressure law of the form \eqref{singpress} was already used in \cite{VH} in the case of a single species, then in \citep{DHV, PDS} for two species.
The PSM cells are mesenchymal, whereas the NT is an epithelial-like tissue where cells are densely packed. These discrepancies in bio-physical properties induce a difference in the viscosity of these two tissues. The Brinkman law \eqref{inibrink} allows to consider these biophysical differences (by considering two different viscosities inside each tissue), as well as to observe rotational movements within the tissues.

\subsection{The mechanical macroscopic models}
\paragraph{Enforced Segregation Viscous Model (ESVM).} We consider two population densities denoted by $n_{1}$ and $n_{2}$ representing respectively the cell density of the NT and the PSM. We endow each tissue with a viscosity parameter, here denoted by $\beta_{1}>0$ and $\beta_{2}>0$, and use the Brinkman law to govern the velocities $v_{1}$ and $v_{2}$ of each tissue. This law takes into account the effect of the viscosity on the pressure within the tissue and the pressure between the tissues (repulsion), thus linking these variables with an elliptic equation. We introduce the following viscous two-species viscous model, for all $(t,x) \in [0; +\infty) \times \mathbb{R}^{d} :$

\begin{eqnarray}
    && \partial_{t}n_{1}+ \nabla \cdot (n_{1}v_{1}) +\alpha \nabla\cdot (n_{1} \nabla  (\Delta n_{1}))=n_{1}G_{1}(p_{1}),
    \label{n1new} \\
   && \partial_{t}n_{2}+ \nabla \cdot(n_{2}v_{2}) +\alpha \nabla \cdot (n_{2} \nabla (\Delta n_{2}))=n_{2}G_{2}(p_{2}),
\label{n2new}\\
    && -\beta_{1}\Delta v_{1}+v_{1}=-\nabla p_{1}, \label{v1new}\\
      && -\beta_{2}\Delta v_{2}+v_{2}=-\nabla p_{2}, \label{v2new}\\
      && p_{1}=p_{\epsilon}(n_{1}+n_{2})+n_{2}q_m(n_{1}n_{2}), \label{p1new}\\
      && p_{2}=p_{\epsilon}(n_{1}+n_{2})+n_{1}q_m(n_{1}n_{2}), \label{p2new}\\
      && q_m(r)=\frac{m}{m-1}((1+r)^{m-1}-1),\; r=n_{1}n_{2},\label{qmnew}\\
    && p_{\epsilon}(n)=\epsilon \frac{n}{1-n},\; n=n_{1}+n_{2}, \label{pnew}
\end{eqnarray}
with $\alpha>0$ a diffusion parameter linked to the width of the interface between the tissues, $m>0$ the parameter controlling the repulsion pressure $q_m$, and $\epsilon>0$ a parameter controlling the congestion pressure $p_\epsilon$. The total pressures $p_1$ and $p_2$ of each tissue (in (\ref{p1new}), (\ref{p2new})) are the sum of the congestion pressure $p_\epsilon$ and the repulsion pressure $q_m$,  where $p_\epsilon$ is a function of the total density $n=n_1+n_2$, and $q_m$ a function of the product $r=n_1 n_2$ which is null whenever $n_1$ and $n_2$ are not in contact. The functions $G_i$ are the growth functions of tissue $i$. They are typically taken decreasing with a zero-value in some pressure value $p_i^\ast$ to take into account homeostasis.\\
We complement the system with initial data,
\begin{eqnarray}
    && \forall x \in \mathbb{R}^{d},\; \;  n_{1}(t=0,x)=n_{1}^{\scriptsize{ini}}(x), \quad  n_{2}(t=0,x)=n_{2}^{\scriptsize{ini}}(x),\label{initdata}
\end{eqnarray}
on which we assume the following conditions,
\begin{eqnarray}
     \forall x \in \mathbb{R}^{d},\; \;  n_{1}^{\scriptsize{ini}}(x) \geq 0, \quad  n_{2}^{\scriptsize{ini}}(x) \geq 0, \label{cond1} \quad \text{and} \quad
     n^{\scriptsize{ini}}(x)\coloneqq   n_{1}^{\scriptsize{ini}}(x)+  n_{2}^{\scriptsize{ini}}(x) < 1,
\end{eqnarray}
where 1 stands for the maximum total density allowed by our singular pressure law \eqref{pnew}. Moreover, given initially positive densities, we can check with a standard Stampacchia method that the model preserves  the positivity of the densities for all times.

The model (\ref{n1new})-(\ref{pnew}) was first introduced in \cite{PDS} in the case where $\beta_1=\beta_2=0$, that is, with Darcy's law for the velocity instead of the Brinkman law. The choice of Brinkman law first introduces boundary conditions on the velocity when we consider the model on a bounded domain in $\mathbb{R}^d,$ giving rise to non trivial curl as a result of the boundary conditions. Second, it introduces a new parameter $\beta$, the viscosity coefficient, which in the case of the NT and the PSM has not been measured yet in literature. The model equations are derived from a gradient flow structure associated to a mechanical energy according to the Wasserstein metric. This energy incorporates terms such as the congestion pressure $p_\epsilon$, the repulsion pressure $q_m$, and a term penalizing high gradients of the densities with a coefficient $\alpha$. The fourth order term compensates for the instabilities caused by the repulsion pressure.

Authors in \citep{PDS} show that this model segregates initially mixed densities in finite time, except for a small interface where mixing is allowed. The width of the mixing region depends on the parameters $\alpha$ and $m$, the repulsion parameter. By introducing the viscosities through the Brinkman law as in the equations~(\ref{v1new}) and (\ref{v2new}), we account for the swirling motions and non trivial curl within the tissues. Interestingly, it was shown in \cite{PV, anyd, Debiec:2020aa, ol} that introducing viscosity when the velocity is in a gradient form induces pressure discontinuities on the boundaries. We prove that this discontinuity happens also when taking the velocity law \eqref{v1new}, \eqref{v2new} in a bounded domain with prescribed boundary conditions. One can also notice similarities between our model and the Cahn-Hilliard equation which is also of fourth order and which promotes the segregation of two phases. On the other hand, we use the singular pressure law~(\ref{pnew}) described to prevent cell overlap, which in the two-species case is a function of the total densities.

In Section \ref{limitres} we will compute the incompressible limit of the ESVM and recover a free boundary problem describing the geometric evolution of the tissue domains. We will look at the quantitative behavior of the ESVM at the incompressible limit in Section \ref{sec:tran_mod1}. Furthermore, we will compare the ESVM to a model where there is no pressure enforcing segregation (that is, $q_m=0$, and accordingly we also take $\alpha=0$ in the ESVM), and with initially segregated densities. This model, referred to as the VM model, is endowed with passive segregation. We describe now the VM model.

\paragraph{Viscous model (VM).}
We consider two population densities $n_{1}$ and $n_{2}$ with their respective velocities $v_{1}$ and $v_{2}$. We use again the Brinkman law for the velocities, with viscosity parameters $\beta_{1}>0$ and $\beta_{2}>0$. Here we consider only the congestion pressure $p_\epsilon$ (and no repulsion pressure). Then, the two-species viscous model (VM) is as follows, for all $(t,x) \in [0; +\infty) \times \mathbb{R}^{d}$,

\begin{eqnarray}
    && \partial_{t}n_{1}+ \nabla \cdot (n_{1}v_{1}) =n_{1}G_{1}(p_\epsilon),
    \label{n1v} \\
   && \partial_{t}n_{2}+ \nabla \cdot(n_{2}v_{2}) =n_{2}G_{2}(p_\epsilon),
\label{n2v}\\
    && -\beta_{1}\Delta v_{1}+v_{1}=-\nabla p_{\epsilon}, \label{v1v}\\
      && -\beta_{2}\Delta v_{2}+v_{2}=-\nabla p_{\epsilon}, \label{v2v}\\
    && p_{\epsilon}=\epsilon \frac{n}{1-n},\; \qquad n=n_{1}+n_{2}, \label{pv}
\end{eqnarray}
where the model parameters satisfy (\ref{cond1}).

This model preserves the segregation property, i.e, initially segregated densities remain segregated at all times: see Proposition \ref{segprop}. For this model, we also compute the incompressible limit and show some quantitative features at the limit.

\paragraph{Boundary conditions.}
For the numerical simulations in the next section, we set our problem in the square $[-1,1]^2$ with boundary conditions that best describe the biological setting. In the embryo, the PSM is surrounded by a solid-like structure or wall (called the lateral plate) which plays the role of the lateral boundaries of our bounded domain. The fluid-like PSM ends at its upper boundary with small solid structures called somites (future vertebrae, see the red rounded squares in Figure \ref{sketch}). The cells adhere at the lateral walls and at the somites. Then, homogeneous Dirichlet boundary conditions on the velocity are a natural choice. Note that they imply zero flux boundary conditions on the density.

For our theoretical analysis in Sections \ref{sec:tran_mod1} and \ref{sec:tran_pres} we consider a more general 2D (smooth) domain with homogeneous Dirichlet boundary conditions.

Note that in Section \ref{sec:proofincomp}, the computations of the incompressible limit of the ESVM are formal and hold in any dimension $d$. For simplicity, we consider the case of the whole space, that is $\mathbb{R}^d$, though the results could be adapted easily to the case of a bounded domain.

\subsection{Numerical illustrations}
In this section we illustrate and compare the models in 2D with numerical simulations in Matlab. The models are considered in a square with zero flux boundary conditions on the densities and homogeneous Dirichlet boundary conditions on the velocities. We use a finite volume semi-implicit scheme on a staggered grid and, for the ESVM, we apply a relaxation method adapted from \cite{PDS} consisting of reducing the order of the equations on the densities from fourth order to second order.

Initially, we consider that part of the NT and of the PSM are already formed. We represent part of the NT (the middle tissue) surrounded by two stripes of PSM (tissues on the left and right of the NT). This initial data is fully segregated with a rather sharp interface between the tissues as it corresponds to the anterior part of the tissues where segregation is the most apparent. More precisely, we take as initial data, ${n_1}^{\scriptsize{ini}}=0.9 \chi_{[-2/3;2/3]\times[-1; 0]}$ and ${n_2}^{\scriptsize{ini}}=0.9\chi_{([-1;-2/3]\cup[[2/3;1])\times [-1;0]} $ with $\chi$ the indicator function (as illustrated in Figure \ref{initialconditionsfigure} \protect \subref{cie}).
 
The PSM being endowed with a higher proliferation rate than the NT \cite{B2017} and with a smaller viscosity coefficient as explained in section \ref{biop}, we choose,
\begin{equation}\label{tissularparam}
G_1(s)=5-s, \quad G_2(s)=10-s, \quad \beta_1=0.5, \quad \beta_2=0.1.
\end{equation}
Finally we pick $\epsilon=0.1$, $m=30$ and $\alpha=0.001$.

We first present a numerical illustration of the ESVM in Figure \ref{sim} Panels \protect \subref{dense}--\protect \subref{curle}.
In Panel \protect \subref{dense} of Figure \ref{sim}, we represent the tissue densities in the ESVM at time $t=0.1$. We see that the NT and the PSM are growing and elongating along the vertical axis which represents the head-to-tail axis of the vertebrate embryo. We clearly see that geometric dynamics emerge, in the sense that the densities remain segregated and each density is either close to zero or to its maximum value (taken equal to 1 as per our pressure law~\eqref{singpress}). This advocates for an incompressible regime.
 As a consequence of its higher proliferation rate and smaller viscosity, the PSM elongates faster than the NT. The tissues co-evolve and remain overall segregated, sharing only a thin interface (where the densities overlay).
\begin{figure}[h]
\centering
\begin{tabular}{cc}
\subfloat[Plot of the densities $n_1(t)$ of the NT (green color bar) and $n_2(t)$ of the PSM (red colorbar) at time $t=0$ in the ESVM.]{\includegraphics[scale=0.35]{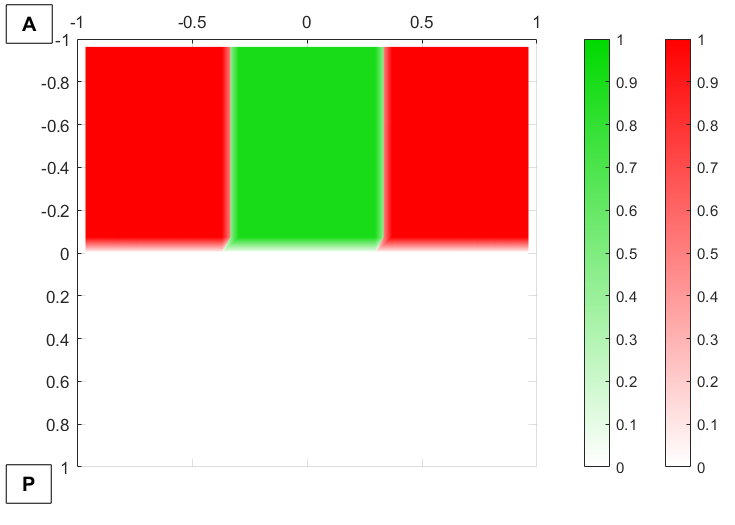} \label{cie} }
   & \subfloat[Plot of the densities $n_1^\infty(t)$ of the NT in green and $n_2^\infty(t)$ of the PSM in red at time $t=0$ in the L-ESVM. The density of each tissue is equal to~1.]{\includegraphics[scale=0.35]{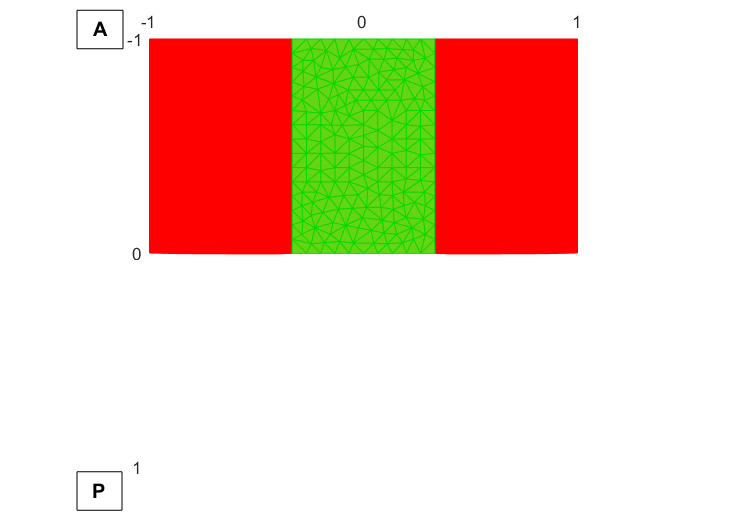}\label{cile}}
   \end{tabular}
      \caption{Initial conditions respectively for the ESVM and the VM (Panel \protect\subref{cie}) and the L-ESVM (Panel \protect \subref{cile}). The notations A and P respectively denote the anterior and the posterior parts of the embryo.}\label{initialconditionsfigure}
   \end{figure}
\begin{figure}[!htbp]
\centering
\begin{tabular}{cc}
\subfloat[PSM and NT densities in the ESVM.
]{\includegraphics[scale=0.35]{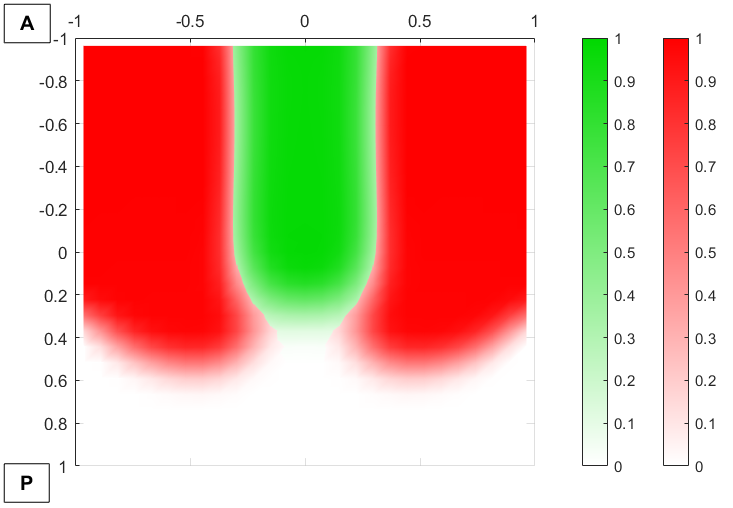}\label{dense}}
   & \subfloat[PSM velocity and curl in the ESVM.
   ]{\includegraphics[scale=0.35]{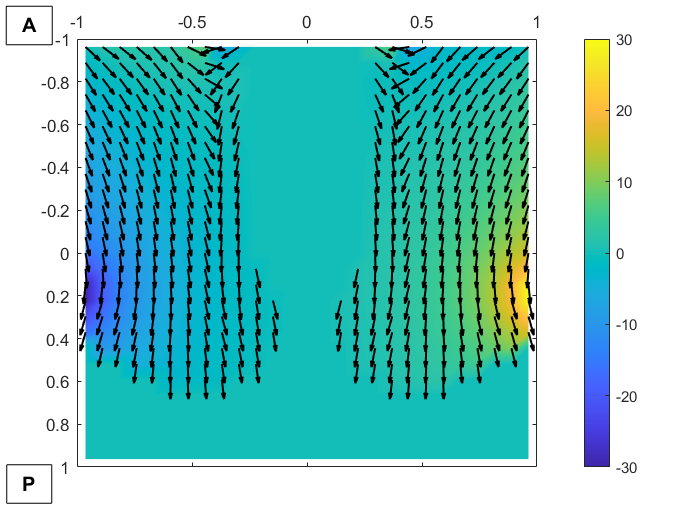}\label{curle}}
   \end{tabular}
   \begin{tabular}{cc}
   \subfloat[PSM and NT densities in the ESVM using \eqref{bg1}-\eqref{bg2} for the velocities.
   ]{\includegraphics[scale=0.35
   ]{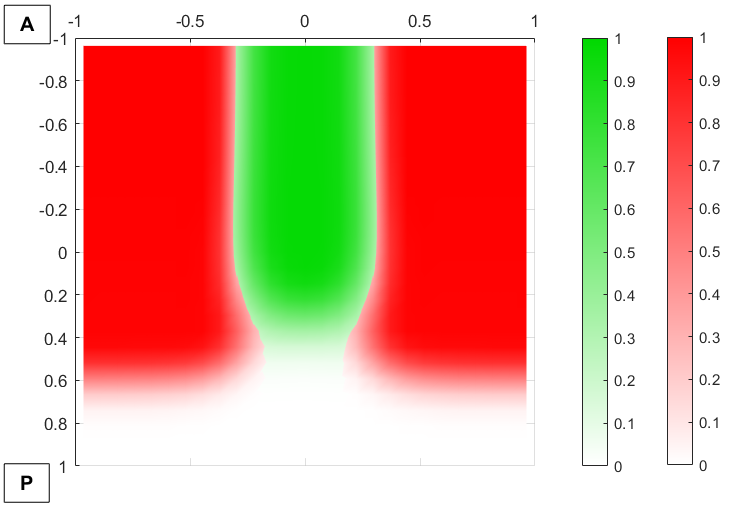}\label{densbg}}
   & \subfloat[PSM velocity and curl in the ESVM using \eqref{bg1}-\eqref{bg2} for the velocities.
   ]{\includegraphics[scale=0.35]{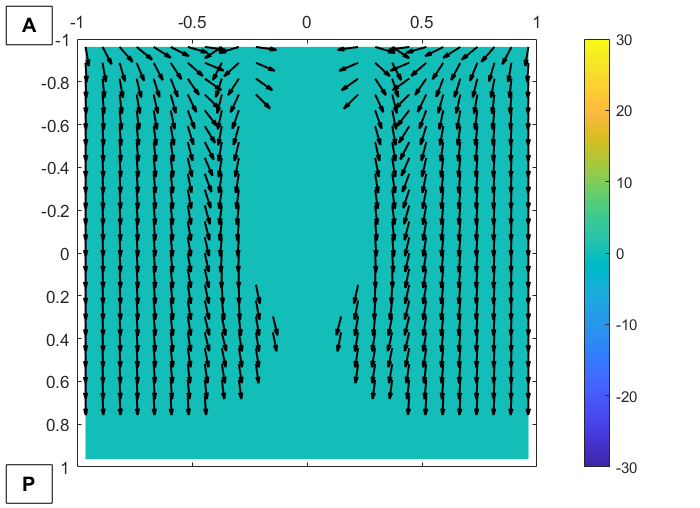}\label{curlbg}} 
   \end{tabular}
   \begin{tabular}{cc}
   \subfloat[PSM and NT densities in the VM.
   ]{\includegraphics[scale=0.35]{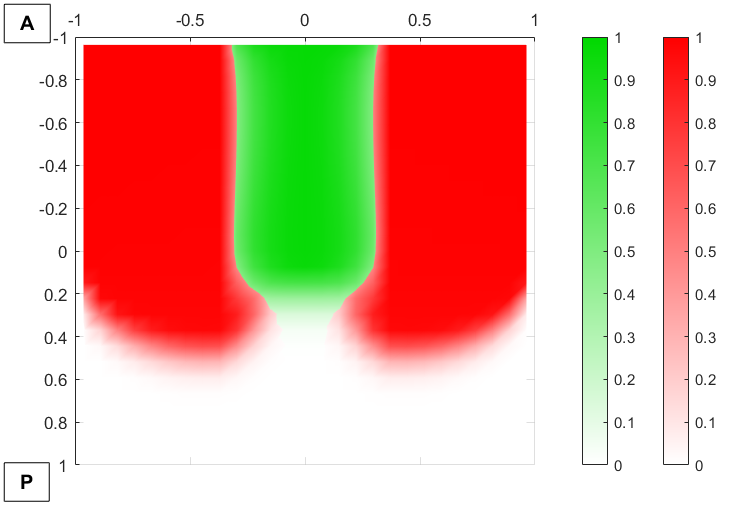}\label{densvm}}
   & \subfloat[PSM velocity and curl in the VM.
   ]{\includegraphics[scale=0.35]{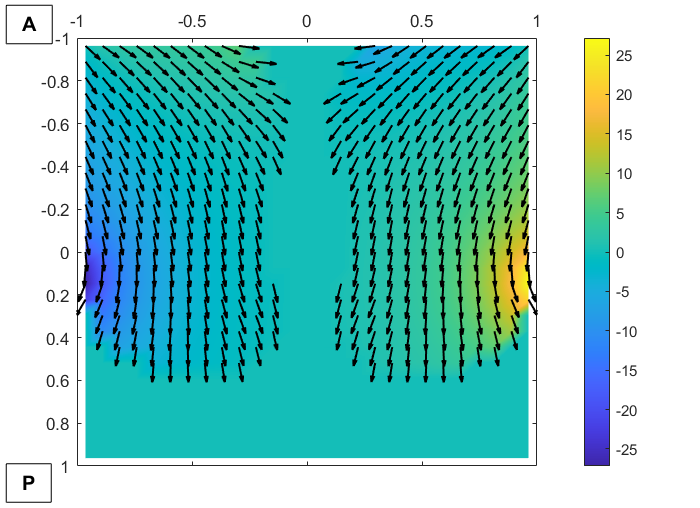}\label{curlvm}} 
   \end{tabular}
   \begin{tabular}{cc}
   \subfloat[PSM and NT densities in the L-ESVM.
   ]{\includegraphics[scale=0.35]{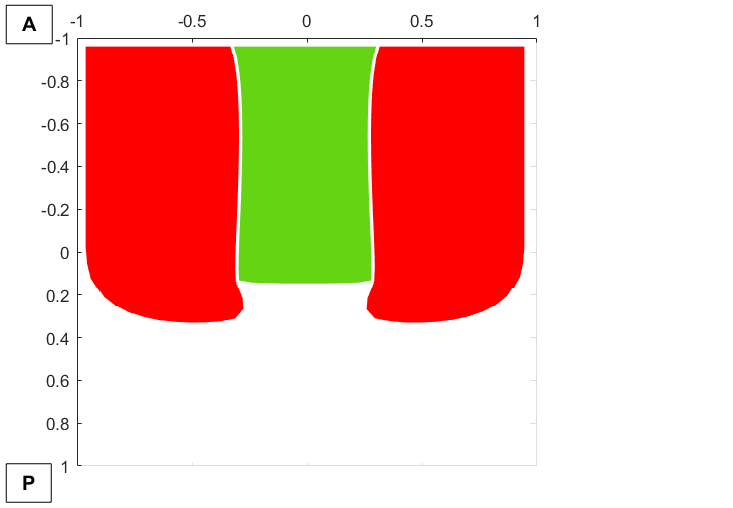}\label{densle}}
   & \subfloat[PSM velocity and curl in the L-ESVM.
   ]{\includegraphics[scale=0.35]{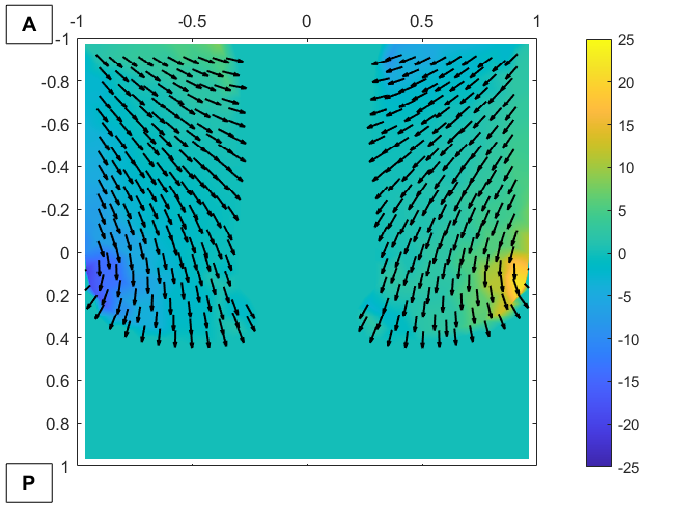}\label{curlle}} 
   \end{tabular}
   \caption{Numerical simulations. 
   The left panels illustrate the NT density $n_1$ (in green) and the PSM density $n_2$ (in red) and the right panels illustrate curl $v_2$ (heat map) overlayed by the vector $\frac{v_2}{\|v_2\|}$ (black arrows) at time $t=~0.1$ in respectively the ESVM (Panel \protect \subref{dense}--\protect \subref{curle}), the ESVM using \eqref{bg1}-\eqref{bg2} for the velocities (Panel \protect \subref{densbg}--\protect \subref{curlbg}), the VM (Panel \protect \subref{densvm}--\protect \subref{curlvm}) and the L-ESVM (Panel \protect \subref{densle}--\protect \subref{curlle}). The velocity vector $\frac{v_2}{\|v_2\|}$ is represented only in the regions where the density of the PSM is above a threshold equal to 0.1.}\label{sim}
   \end{figure}

In Figure \ref{sim} Panel \protect \subref{curle}, we illustrate that the ESVM produces a non-zero curl of the PSM velocity. Using the solutions $(v_1,v_2)$ of the equations \eqref{v1new}-\eqref{v2new} in the ESVM, we compute the curl inside each tissue, and show the result for the velocity $v_2$ (only in the PSM) in Panel \protect \subref{curle} where the black arrows are the velocity vectors and the heatmap represents the value of the curl. High curls are observed close to the lateral boundaries, especially in the posterior zone. This corresponds to the migration of cells from the PZ into the posterior PSM as described in Section \ref{intro}, with clockwise vortices (negative curl) on the left wall and symmetrically counter-clockwise vortices (positive curl) on the right wall. In the posterior part, close to these lateral walls the velocity vectors clearly display swirling fluid motion. These curls are propagated inside the PSM and along the antero-posterior axis. Anteriorly, we observe some adjacent regions of opposite curls on both sides of the interface. This is in accordance with the observations in the embryo, see Figure \ref{sketch}. Finally, in the anterior region the velocity vectors appear to cross directions, suggesting tissue contractions near the interface. Such contraction is due to the repulsion force $q_m$.

We now present in Figure \ref{sim} Panels \protect \subref{densbg}--\protect \subref{curlbg} a numerical illustration of the ESVM in the case where the velocity is taken of gradient form in the Brinkman law as in \eqref{inibrinkgrad}. That is, we change the velocity laws used in the ESVM \eqref{v1new}-\eqref{v2new} (coupled with Dirichlet boundary conditions) and replace them with the following equations,
\begin{equation}
v_1=-\nabla K_1, \quad \text{ with } -\beta_1\Delta K_1, +K_1=p_1\label{bg1}
\end{equation}
\begin{equation}
v_2=-\nabla K_2, \quad \text{ with } -\beta_2\Delta K_2+K_2=p_2,\label{bg2}
\end{equation}
where $K_1$ and $K_2$ are potentials (coupled with homogeneous Neumann boundary conditions for the potentials $K_i$, $i=1,2$).
 All the other equations in the ESVM (equations on the densities and the expressions of the pressures) remain unchanged. Our choice of parameters also remains unchanged and is taken as in \eqref{tissularparam}. We solve this new system and compute the curl of the velocities $v_1$ and $v_2$ using the equations \eqref{bg1}-\eqref{bg2}. In Figure \ref{sim}\protect \subref{curlbg} we represent curl $v_2$ by the heatmap overlayed by the velocity vector $v_2$ (the arrows in black).
In contrast with the ESVM, when the velocity is taken as a gradient, we observe a laminar flow: the curl is null everywhere and the vector trajectories are rather straight and pointing towards the posterior. Figures \ref{sim}\protect \subref{dense}-\protect \subref{curle} and \protect \subref{densbg}-\protect \subref{curlbg} show two different dynamics emerging from the two laws considered for the velocity: one is laminar (Figure \ref{sim}\protect \subref{curlbg}) and the other displays swirling motion (Figure \ref{sim}\protect \subref{curle}). Overall, the results of the ESVM describe more accurately the tissue evolution and cell movements observed in the vertebrate embryo, schematically represented in Figure \ref{sketch}\protect \subref{sketchrot}.

We finally present a numerical illustration of the VM in Figure \ref{sim} Panels \protect \subref{densvm}--\protect \subref{curlvm} and compare it with the ESVM.
In the VM, as in the ESVM, we observe high curls in the posterior zone, especially close to the lateral boundaries. However, in contrast to the ESVM, we no longer see in the anterior region adjacent zones of opposite curls but instead we observe unidirectional latero-medial curls (counter clockwise on the left of the NT and clockwise on the right). Both models recover well the swirling motion in the posterior zone. Furthermore, each model, the ESVM and the VM, shows a specific feature observed in the anterior zone of the vertebrate embryo. The ESVM highlights the role of the repulsion pressure in creating alternating zones of opposite curls anteriorly, while the VM displays the global latero-medial curls reported in the anterior region of the PSM, as represented in Figure \ref{sketch}\protect \subref{sketchrot}. A sensitivity analysis on the model parameters (especially those affecting the repulsion force) would allow to obtain a representation of the swirling motions as close as possible to the biological data. A future work will be dedicated to this sensitivity analysis.

 The last set of simulations displayed in Figure \ref{sim}
(Panels \protect \subref{densle} and \protect \subref{curlle}) will be commented in Section~\ref{limitres} below. 


\section{Main results}\label{sec:mainres}

\subsection{\hyperref[sec:proofincomp]{Incompressible limit of the ESVM}}\label{limitres}
Our first result is the formal incompressible limit of the two-species viscous model, where we take the parameters $\epsilon \rightarrow 0, \alpha \rightarrow 0$ and $m\rightarrow+\infty$. This allows us to derive a free boundary problem, as detailed in the following theorem.

\begin{theorem}(Formal)\label{formalim}[Incompressible limit of the ESVM]
Let $n^{ini}_{1},\, n^{ini}_{2}$ satisfy the conditions \eqref{cond1} and let $n_{1}$, $n_{2}$, $v_{1}$, $v_{2}$, $p_{1}$, $p_{2}$ solve the viscous two-species model \eqref{n1new}-\eqref{p2new} with \eqref{qmnew}--\eqref{initdata}. Assume that at the incompressible limit, that is, when $\epsilon$ and $\alpha$ go to zero and $m$ goes to infinity, the quantities $n_{1}$, $n_{2}$, $v_{1}$, $v_{2}$, $p_{1}$, $p_{2}$, $q_m(n_1 n_2)$ converge (in a sufficiently strong sense) towards, respectively,  $n_{1}^{\infty}$,  $n_{2}^{\infty}$, $v_{1}^{\infty}$,  $v_{2}^{\infty}$, $p_{1}^{\infty}$, $p_{2}^{\infty}$, $q^{\infty}$. Then, these quantities satisfy the following system of equations for all $(t,x) \in [0; +\infty)\times \mathbb{R}^{d}$,
\begin{eqnarray}
    && \partial_{t}n_{1}^{\infty}+ \nabla \cdot (n_{1}^{\infty}v_{1}^{\infty}) =n_{1}^{\infty}G_{1}(p_{1}^{\infty}),\label{nn1}
     \\ 
   && \partial_{t}n_{2}^{\infty}+ \nabla \cdot (n_{2}^{\infty}v_{2}^{\infty}) =n_{2}^{\infty}G_{2}(p_{2}^{\infty}),\label{nn2}
\\
    && -\beta_{1}\Delta v_{1}^{\infty}+v_{1}^{\infty}=-\nabla p_{1}^{\infty}, \label{vv1}\\
      && -\beta_{2}\Delta v_{2}^{\infty}+v_{2}^{\infty}=-\nabla p_{2}^{\infty}, \label{vv2}\\
      && p_{1}^{\infty}=p^{\infty}+n_{2}^{\infty}q^{\infty}, \\
      && p_{2}^{\infty}=p^{\infty}+n_{1}^{\infty}q^{\infty}, \label{pp1}
\end{eqnarray}
and the following relation holds,
\begin{equation}
    p^{\infty}(1-n^{\infty})=0, \; \text{where} \; n^{\infty}=n_{1}^{\infty}+n_{2}^{\infty}. \label{cong}
\end{equation}
Moreover, we obtain full segregation of the two-species at the limit,
\begin{equation}
    n_{1}^{\infty}n_{2}^{\infty}=0. \label{segr}
\end{equation}
The complementary relation prescribing the dynamics of the pressures inside the two tissues at the limit reads,
\begin{eqnarray}
        p^{\infty^{2}}\bigg(\nabla \cdot \left( n_{1}^{\infty} v_{1}^{\infty}\right)+ \nabla \cdot \left( n_{2}^{\infty}v_{2}^{\infty}\right)\bigg)&=&
        p^{\infty^2}\bigg(n_{1}^{\infty}G_{1}(p_{1}^{\infty})+n_{2}^\infty G_{2}(p_{2}^{\infty})\bigg).
        \label{congpress}
\end{eqnarray}
Finally, defining $K_1^{\infty}=n_1^{\infty}q^{\infty}$ and $K_2^{\infty}=n_2^{\infty}q^{\infty}$, the equations prescribing the dynamics of the pressure due to repulsion at the limit read,
\begin{eqnarray}
\partial_t K_1^{\infty}&+&\nonumber (q^{\infty}+1)\log(q^{\infty}+1)\nabla \cdot (n_1^{\infty}v_1^{\infty})+(q^{\infty}+1)\log(q^{\infty}+1)n_1^{\infty}\nabla \cdot v_2^{\infty}
\\
\nonumber
&+&n_1^{\infty} v_2^{\infty} \cdot   \nabla q^{\infty} - (q^{\infty}+1)\log(q^{\infty}+1)\nabla n_1^{\infty}\cdot v_2^{\infty} +q^{\infty}\nabla \cdot (n_1^{\infty}v_1^{\infty})
\nonumber
\\
&=& (q^{\infty}+1)\log(q^{\infty}+1)n_1^{\infty}(G_1(p_1^{\infty})+G_2(p_2^{\infty}))+q^{\infty}n_1^{\infty}G_1(p_1^{\infty}),\label{u1inf}
\end{eqnarray}

\begin{eqnarray}
\partial_t K_2^{\infty}&+&\nonumber (q^{\infty}+1)\log(q^{\infty}+1)\nabla \cdot (n_2^{\infty}v_2^{\infty})+(q^{\infty}+1)\log(q^{\infty}+1)n_2^{\infty}\nabla \cdot v_1^{\infty} \\\nonumber
&+&n_2^{\infty} v_1^{\infty} \cdot   \nabla q^{\infty} - (q^{\infty}+1)\log(q^{\infty}+1)\nabla n_2^{\infty}\cdot v_1^{\infty} +q^{\infty}\nabla \cdot (n_2^{\infty}v_2^{\infty}) \nonumber\\
&=& (q^{\infty}+1)\log(q^{\infty}+1)n_2^{\infty}(G_1(p_1^{\infty})+G_2(p_2^{\infty}))+q^{\infty}n_2^{\infty}G_2(p_2^{\infty}).\label{u2inf}
\end{eqnarray}
We call L-ESVM the limiting system \eqref{nn1}--\eqref{u2inf} thus obtained.
\end{theorem}

This formal limit being established, we consider the situation where each tissue $i=1,2$ occupies exactly and fully a specific (moving) domain $\Omega_i(t)$, that is, the density $n_i^\infty(t,\cdot)$ is the indicator function of the domain $\Omega_i(t)$. Note that the congestion relation \eqref{cong} imposes that the pressure $p^\infty$ is null outside $\Omega_1\cup\Omega_2$, and that the full segregation \eqref{segr} imposes that $\Omega_1$ and $\Omega_2$ do not intersect.

From System \eqref{nn1}--\eqref{u2inf}, we can then deduce the evolution of the domains $\Omega_1$ and $\Omega_2$, thus giving rise to a geometrical description of the biological system. Indeed, the velocities of the exterior boundary of the first tissue's domain, the exterior boundary of the second tissue's domain and the interface of the two tissues' domains are given by $(v_1^\infty \cdot \vec{\nu}) \vec{\nu} $, by $(v_2^\infty \cdot \vec{\mu}) \vec{\mu}$, and by $(v_1^\infty \cdot \vec{\nu}) \vec{\nu} =(v_2^\infty \cdot \vec{\mu}) \vec{\mu}$, respectively, where $\vec{\nu}$ is the outward normal vector to $\Omega_1$ and $\vec{\mu}$ is the outward normal vector to $\Omega_2$ (on the interface between $\Omega_1$ and $\Omega_2$ we have $\vec{\nu}=-\vec{\mu}$ ).
Such situation is represented in Figure \ref{fig:schema}.

We perform numerical simulations in Freefem++ (Figure \ref{sim}\protect \subref{densle}-\protect \subref{curlle}) to illustrate  the evolution of the free boundary problem in the case of the embryo, that is, a NT surrounded by two stripes of PSM. The initial conditions were taken as two indicator functions sharing a sharp interface, ${n_1}^{\scriptsize{\infty,ini}}=\chi_{[-2/3;2/3]\times[-1; 0]}$ and ${n_2}^{\scriptsize{\infty,ini}}=\chi_{([-1;-2/3]\cup[[2/3;1])\times [-1;0]} $ with $\chi$ the indicator function, and we take $q^{\infty,ini}=0$ (as illustrated in Figure \ref{initialconditionsfigure}\protect \subref{cile}). Note that, as sketched in Figure \ref{fig:link}, taking the initial densities segregated and the repulsion force $q^{\infty,ini}=0$, the L-ESVM in fact coincides with the L-VM. Therefore, our numerical simulations illustrate both the L-ESVM and the L-VM.
The parameters are taken as in the numerical illustrations of the ESVM, see \eqref{tissularparam}. 
In Figure \ref{sim}\protect \subref{densle} we show the evolution of the free domains of the NT and of the PSM at $t=0.1$. The tissues evolve while remaining completely segregated with a sharp interface throughout the simulation. The PSM, which has a higher proliferation rate and a smaller viscosity coefficient, tends to grow faster than the NT (central tissue) and occupies a wider space. Comparing Figure \ref{sim}\protect \subref{densle} with Figures \ref{sim}\protect \subref{dense} and \ref{sim}\protect \subref{densvm} we see that the density profiles are qualitatively similar. This is in accordance with our choice of parameters $\alpha$, $m$ (in the ESVM) and $\epsilon$ (in the ESVM and VM) taken in their asymptotic ranges.

In Figure \ref{sim}\protect \subref{curlle} we illustrate the curl of the PSM velocity computed with the L-ESVM. High curls are observed in the posterior zone close to the lateral walls, more precisely, negative curls on the left and positive curls on the right. Anteriorly, opposite curls are observed at the interface between the tissues. These profiles match very well those observed in the VM (Figure \ref{sim}\protect \subref{curlvm}) where the repulsion force $q_m$ is absent. This is in accordance with our choice $q^\infty=0$. Finally, we exhibit the velocity vector $v_2^\infty$ in the L-ESVM. In the posterior zone, clear rotating vector trajectories are observed on the lateral walls where curls are high. Overall, our illustrations show that the three models VM, ESVM and L-ESVM reproduce the swirling motions observed in the embryo (Figure \ref{sketch}\protect \subref{sketchrot}). However, only the ESVM (in the asymptotic regime) was able to reproduce the adjacent zones with opposite curls observed anteriorly in the embryo. This suggests that the repulsion pressure does not vanish at the incompressible limit and plays an important biophysical role at the limit as discussed in Section \ref{section2.3}. It also suggests that some active segregation may be at play to maintain the tissues segregated.
Finally our simulation of the L-ESVM displays similar dynamics as that of the ESVM (and the VM) in the asymptotic regime which demonstrates the relevance of the incompressible limit.

\begin{figure}[h!]
  \centering
  \includegraphics[scale=0.35]{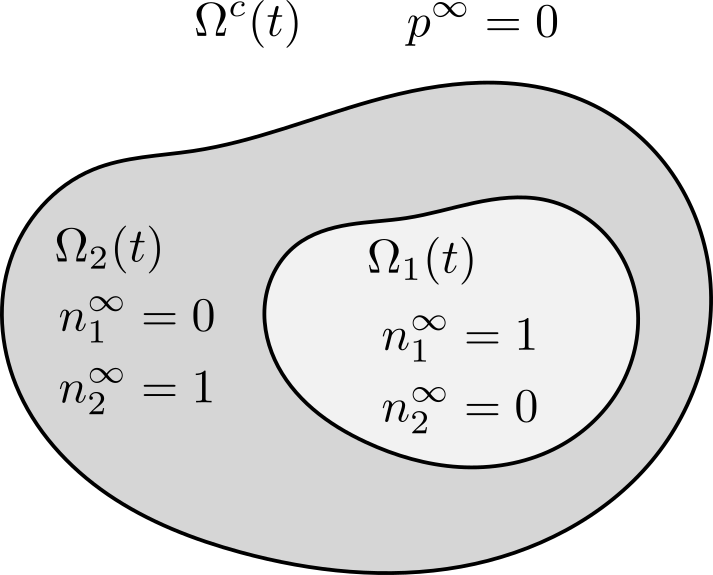}
  \caption{Representation of the subdomains saturated with the two densities in the L-ESVM.}
  \label{fig:schema}
  \end{figure}

\subsection{\hyperref[sec:tran_mod1]Study of the stationary L-ESVM ($q$ given)}\label{sec:studystatLESVM}

Still in the case when the densities are indicator functions (at the incompressible limit), we are now interested in the study of the stationary velocity-pressure system. For clarity we now remove the superscript "$\infty$" from the variables $n_i,\,v_i,\,q,\,p$ for $i=1\,2$.

We work in a bounded domain $\Theta \subset \mathbb{R}^{2}$, and we consider two stationary subdomains $\Omega_1$ and $\Omega_2$. We suppose that $\Theta$, $\Omega_1$ and $\Omega_2$ are as follows,
\begin{eqnarray}
&&\Theta\text{ a smooth bounded domain in } \mathbb{R}^{2}, \Omega, \, \Omega_1, \, \Omega_2 \text{ subdomains of } \Theta \text{ such that } \label{cc1} \\ 
&& \Omega_1 \cap \Omega_2  \text{ is empty,} \; \Gamma\coloneqq \overline{\Omega_1} \cap \overline{\Omega_2} \text{ is non-empty,} \label{cc1b} \\ 
&& \Omega=\Omega_1\cup \Omega_2 \text{ and } \overline\Omega \subset \Theta, \label{cc2}\\
&& \text{and, defining }\Omega^{c}\coloneqq  \Theta \backslash \overline{\Omega}, \text{ the boundaries } \Gamma,\; \Gamma_1\coloneqq \overline{\Omega_1} \cap \overline{\Omega^{c}}, \; \Gamma_2=\overline{\Omega_2} \cap \overline{\Omega^{c}} \text{ are  } \mathcal{C}^{\infty}.\label{cc3}
\end{eqnarray}

Our analysis is conducted in a simplified framework, where we drop the equations \eqref{u1inf} and \eqref{u2inf} and consider instead that $q$ is a given function on $\Omega$. We assume,
\begin{equation}\label{qL2}
q\in L^2(\Omega).
\end{equation}

Finally, for simplicity, we suppose that the growth function is linear (and decreasing), that is,
\begin{equation}\label{lineargrowth}
G_1(s)=g_1(p_1^{\ast}-s), \quad G_2(s)=g_2(p_2^{\ast}-s),
\end{equation}
for some $g_1, \, g_2, \, p_1^\ast, \, p_2^\ast>0$.

In this framework, the system \eqref{nn1}--\eqref{congpress} at equilibrium and complemented with homogeneous Dirichlet boundary conditions becomes the following elliptic system on $(v_1,v_2)$, for $q$ given as in \eqref{qL2},

\begin{equation}
(S_2) \;\; 
\left\{
\begin{array}{ll}
-\beta_{1}\Delta v_{1}+v_{1}=- \nabla [(p_1^{\ast}-\frac{1}{g_1}\nabla \cdot v_1)\chi_{\Omega_1}+(p_2^{\ast}-\frac{1}{g_2}\nabla \cdot v_2+q)\chi_{\Omega_2}]& \text{\; on \;} \Theta,  \\
-\beta_{2}\Delta v_{2}+v_{2}=- \nabla [(p_2^{\ast}-\frac{1}{g_2}\nabla \cdot v_2)\chi_{\Omega_2}+(p_1^{\ast}-\frac{1}{g_1}\nabla \cdot v_1+q)\chi_{\Omega_1}]& \text{\; on \;} \Theta, \\
    v_1=v_2=0& \text{\; on \;} \partial \Theta, 
    \end{array}
     \right.
\nonumber
\end{equation}
For the existence theory and the elliptic regularity, we assume that the model parameters satisfy the following condition,
\begin{eqnarray}
\beta_1 g_2 >\frac{1}{4}, \qquad \text{and} \qquad \beta_2 g_1 >\frac{1}{4}. \label{cdd1}
\end{eqnarray}
\begin{notation}
Let $U\subset \mathbb{R}^2$ be a bounded domain. For $k\ge 0$, we use the notation $\mathcal{C}^k(\overline{U})$, or equivalently write `` $\mathcal{C}^k$ up to the boundary of $U$", in the classical sense, see \cite{evans}:
\begin{equation}
\mathcal{C}^k(\overline{U})=\{u \in \mathcal{C}^k(U)| D^\alpha u\; \text{ is uniformly continuous on $U$ for all $|\alpha|\leq k$} \}.\nonumber
\end{equation}
Thus if $u \in \mathcal{C}^k(\overline{U})$, then $D^\alpha u$ can be continuously extended to $\overline{U}$ for each multi-index $\alpha$, with $|\alpha|\leq k$.

For $\mu\in(0,1)$, we use the notation $\mathcal{C}^{0,\mu}(\overline{U})$ for (uniformly) H\"older continuous functions on $U$ with exponent $\mu$. Thus, such functions can be extended into a (uniformly) H\"older continuous functions on $\overline{U}$ with the same exponent $\mu$.

Finally, for $k\ge 0$ and $\mu\in(0,1)$, we write,
\begin{equation}
\mathcal{C}^{k,\mu}(\overline{U})=\{u \in \mathcal{C}^k(U)| D^\alpha u\;\in \mathcal{C}^{0,\mu}(\overline{U})  \text{ for all $|\alpha|\leq k$} \}.\nonumber
\end{equation}
\end{notation}

\begin{theorem}\label{wellpreg} [Well-posedness and regularity, stationary L-ESVM] Let $\Theta$, $\Omega_1$ and $\Omega_2$ be bounded domains as in (\ref{cc1})--(\ref{cc3}), and let $\beta_1, \beta_2, g_1,g_2>0$ satisfy the condition (\ref{cdd1}). Let $q$ as in \eqref{qL2}. Then, there exists a unique solution $(v_1,v_2) \in H^{1}_0(\Theta)^{2}\times H^{1}_0(\Theta)^2$ to System $(S_2)$. Note in particular that $(v_1,v_2)$ is continuous across the interfaces $\Gamma_1,\Gamma_2,\Gamma$.

Furthermore, if $q\in \mathcal{C}^{0,\mu}(\overline{\Omega_1})\cap \mathcal{C}^{0,\mu}(\overline{\Omega_2})$ for some $\mu\in(0,1)$, then the solution $(v_1,v_2)$ to System $(S_2)$ lies in $\mathcal{C}^{1,\alpha'}(\overline{\Omega_1})\cap\mathcal{C}^{1,\alpha'}( \overline{\Omega_2}) \cap \mathcal{C}^{1,\alpha'}( \overline{\Omega^c})$, with $0<\alpha'\leq\min(\mu,\frac{1}{4})$.
\end{theorem}


\begin{remark}\label{rem:q_out_Omega} Note that $q$ is defined only in $\Omega$, so that we do not need any information on $q$ in regions where both densities are simultaneously equal to zero.
\end{remark}
\begin{remark} The result can be extended to the case of (smooth) non-homogeneous Dirichlet boundary conditions.
\end{remark}
\begin{remark} Since $q=0$ is always a particular solution of (the stationary versions of) Eq. \eqref{u1inf} and \eqref{u2inf} independently of $(v_1,v_2)$, System $(S_2)$ with $q=0$ gives a solution to the original system L-ESVM taken at equilibrium. Note that the case where $q=0$ on $\Omega$ is included in the well-posedness and regularity results.\label{qzero}
\end{remark}

We can now rewrite system $(S_2)$ as a transmission problem, giving rise to non-trivial transmission conditions at the interfaces $\Gamma$, $\Gamma_1$ and $\Gamma_2$.
To this end we introduce the following notation: for $D$ a domain of $\mathbb{R}^2$ and $h$ a continuous function on $D$ that can be continously extended on $\overline{D}$, then, for $x\in \partial D$, we note
\begin{equation} (h)_{D}(x)\coloneqq \lim_{\substack{y\in D,\\ y\rightarrow x}} h(y). \label{def_h_D}
\end{equation}

\begin{proposition}[Transmission problem, stationary L-ESVM]\label{prop:TP}
Assume (\ref{cc1})--(\ref{cc3}), (\ref{cdd1}), and $q\in \mathcal{C}^{0,\mu}(\overline{\Omega_1})\cap \mathcal{C}^{0,\mu}(\overline{\Omega_2})$ for some $\mu\in(0,1)$. Then the solution of system $(S_2)$ solves the following transmission problem $(T_2)$ considered on $\Theta$ and coupled with homogeneous Dirichlet boundary conditions on $\partial \Theta$,

\hspace{-0.5cm}
\begin{tabular}{c}
$(T_{2,\Omega_1})
\left\{
  \begin{array}{ll}
  -\beta_1 \Delta v_1+v_1-\frac{1}{g_1}\nabla \nabla \cdot v_1=0 & \text{ in }\Omega_1, \\
  -\beta_2 \Delta v_2+v_2-\frac{1}{g_1}\nabla \nabla \cdot v_1=-\nabla q & \text{ in } \Omega_1.
  \end{array}
\right.$\vspace{0.3cm}\\

$(T_{2,\Omega_2})\left\{
  \begin{array}{ll}
  -\beta_1 \Delta v_1+v_1-\frac{1}{g_2}\nabla \nabla \cdot v_2=-\nabla q & \text{ in } \Omega_2, \\
  -\beta_2 \Delta v_2+v_2-\frac{1}{g_2}\nabla \nabla \cdot v_2=0 & \text{ in } \Omega_2.
  \end{array}
\right.$
\end{tabular}\vspace{0.3cm}

\hspace{-0.5cm}
\begin{tabular}{c}
$(T_{2,\Omega^c}) \left\{
  \begin{array}{ll}
  -\beta_1 \Delta v_1+v_1=0  & \text{ in } \Omega^c,\\
-\beta_2 \Delta v_2+v_2=0  & \text{ in } \Omega^c.\\
  \end{array}
\right.$
\end{tabular}\vspace{0.3cm}

\hspace{-0.5cm}
\begin{tabular}{c}
$(T_{2,\Gamma_1})\left\{
\begin{array}{ll}
\beta_1 [(\nabla v_1)_{\Omega_1}-(\nabla v_1)_{\Omega^c}]\cdot \vec{\nu}=[p_1^{\ast}-\frac{1}{g_1}(\nabla \cdot v_1)_{\Omega_1}] \vec{\nu}  & \text{ on } \Gamma_1,\\
\beta_2 [(\nabla v_2)_{\Omega_1}-(\nabla v_2)_{\Omega^c}]\cdot \vec{\nu}=[p_1^{\ast}+(q)_{\Omega_1}-\frac{1}{g_1}(\nabla \cdot v_1)_{\Omega_1}] \vec{\nu} & \text{ on } \Gamma_1,\\
(v_1)_{\Omega_1}=(v_1)_{\Omega^c}, \quad (v_2)_{\Omega_1}=(v_2)_{\Omega^c}  & \text{ on } \Gamma_1.
\end{array}
\right.$\\
\end{tabular}\vspace{0.3cm}

\hspace{-0.5cm}
\begin{tabular}{c}
$(T_{2,\Gamma_2})\left\{
\begin{array}{ll}
\beta_1 [(\nabla v_1)_{\Omega_2}-(\nabla v_1)_{\Omega^c}]\cdot \vec{\mu}
=[p_2^{\ast}+(q)_{\Omega_2}-\frac{1}{g_2}(\nabla \cdot v_2)_{\Omega_2}]\vec{\mu} & \text{ on } \Gamma_2,\\
\beta_2 [(\nabla v_2)_{\Omega_2}-(\nabla v_2)_{\Omega^c}]\cdot \vec{\mu}
=[p_2^{\ast}-\frac{1}{g_2}(\nabla \cdot v_2)_{\Omega_2}]\vec{\mu} & \text{ on } \Gamma_2,\\
(v_1)_{\Omega_2}=(v_1)_{\Omega^c}, \quad (v_2)_{\Omega_2}=(v_2)_{\Omega^c}  & \text{ on } \Gamma_2.
\end{array}
\right.$\\
 \end{tabular}\vspace{0.3cm}
 
\hspace{-0.5cm}
\begin{tabular}{c}
$(T_{2,\Gamma})\left\{
\begin{array}{ll}
\beta_1 [(\nabla v_1)_{\Omega_1}-(\nabla v_1)_{\Omega_2}]\cdot \vec{\nu}=[(p_1^{\ast}-p_2^{\ast})-(q)_{\Omega_2}+\frac{1}{g_2}(\nabla \cdot v_2)_{\Omega_2}-\frac{1}{g_1}(\nabla \cdot v_1)_{\Omega_1}] \vec{\nu} & \text{ on } \Gamma,\\
\beta_2 [(\nabla v_2)_{\Omega_1}-(\nabla v_2)_{\Omega_2}]\cdot \vec{\nu}=[(p_1^{\ast}-p_2^{\ast})+(q)_{\Omega_1}+\frac{1}{g_2}(\nabla \cdot v_2)_{\Omega_2}-\frac{1}{g_1}(\nabla \cdot v_1)_{\Omega_1}]\vec{\nu} & \text{ on } \Gamma,\\
(v_1)_{\Omega_1}=(v_1)_{\Omega_2}, \quad (v_2)_{\Omega_1}=(v_2)_{\Omega_2},\quad v_1\cdot \vec{\nu}=v_2\cdot \vec{\nu}  & \text{ on } \Gamma.
\end{array}
\right.$
\end{tabular}
\end{proposition}


To obtain more regularity on the solution, we assume the following additional condition on the domains:
\begin{eqnarray}
&&\partial \Omega \text{ is smooth } (\mathcal{C}^{\infty}). \label{cc4}
\end{eqnarray}
A typical situation where the domains satisfy the assumption above is the situation where one tissue is encompassed within the other one, such as represented in Fig. \ref{fig:schema}. We obtain the following result.

\begin{theorem}\label{reg+} [Further regularity, two concentric species] Let $\Theta$, $\Omega_1$ and $\Omega_2$ be bounded domains as in (\ref{cc1})--(\ref{cc3}), which satisfy furthermore \eqref{cc4}. Let $\beta_1, \beta_2, g_1,g_2>0$ satisfy the condition (\ref{cdd1}). Let $q\in \mathcal{C}^{\infty}(\overline{\Omega_1})\cap \mathcal{C}^{\infty}(\overline{\Omega_2})$, then the solution $(v_1,v_2)$ to System $(S_2)$ lies in $\mathcal{C}^{\infty}(\overline{\Omega_1})\cap\mathcal{C}^{\infty}( \overline{\Omega_2}) \cap \mathcal{C}^{\infty}( \overline{\Omega^c})$.
\end{theorem}

\subsection{\hyperref[sec:tran_pres]{Enforced segregation and propagated segregation.}}

We recall that in the VM, defined in \eqref{n1v}--\eqref{pv}, there is no enforced segregation, contrarily to the ESVM. However, the VM is endowed with the segregation property, i.e, initially segregated densities remain segregated at all times. This is stated in the following proposition.
 
\begin{proposition}\label{segprop}(Formal) [Segregation property, VM]
Let $ n_{1},\, n_{2}$ solve \eqref{n1v}--\eqref{pv}. If the initial densities are fully segregated, that is, $ n_{1}^{ini}n_{2}^{ini}=0$ on $\mathbb{R}^d$, then the densities remain segregated for all times: 
\begin{equation}
n_{1}n_{2}(t,x)=0, \; \forall (t,x) \in [0; +\infty) \times \mathbb{R}^d.\label{fullsegeps}
\end{equation}

\end{proposition}

The incompressible limit of the VM is formally obtained by taking $\epsilon \rightarrow{}0$. Then we obtain the following system L-VM, for all $(t,x) \in [0; +\infty)\times \mathbb{R}^{d} : $
\begin{eqnarray}
    && \partial_{t}n_{1}^{\infty}+ \nabla \cdot (n_{1}^{\infty}v_{1}^{\infty}) =n_{1}^{\infty}G_{1}(p^{\infty}),
     \label{n1lim}\\
   && \partial_{t}n_{2}^{\infty}+ \nabla \cdot (n_{2}^{\infty}v_{2}^{\infty}) =n_{2}^{\infty}G_{2}(p^{\infty}),\label{n2lim}
\\
    && -\beta_{1}\Delta v_{1}^{\infty}+v_{1}^{\infty}=-\nabla p^{\infty}, \\
      && -\beta_{2}\Delta v_{2}^{\infty}+v_{2}^{\infty}=-\nabla p^{\infty}, 
\end{eqnarray}
and the following relation holds
\begin{equation}
    p^{\infty}(1-n^{\infty})=0, \; \text{where} \; n^{\infty}=n_{1}^{\infty}+n_{2}^{\infty}. \label{relationpn}
\end{equation}
The complementary relation prescribing the dynamics of the pressure due to congestion at the limit reads
\begin{eqnarray}
        p^{\infty^{2}}\bigg(n_{1}^{\infty} \nabla \cdot v_{1}^{\infty}+n_{2}^{\infty} \nabla \cdot v_{2}^{\infty}\bigg)&=&
        p^{\infty^2}\bigg(n_{1}^{\infty}G_{1}(p^{\infty})+n_{2}^\infty G_{2}(p^{\infty})\bigg).\label{plim}
\end{eqnarray} 

\begin{remark}
Using the same method as for the VM, we can show the segregation property holds for the L-VM, that is, when initially segregated, the tissues remain segregated for all times:
\begin{equation}
n_1^{\infty}n_2^{\infty}=0, \; \forall (t,x) \in [0; +\infty) \times \mathbb{R}^d. \label{seg}
\end{equation}
Such property can also be directly inherited from the segregation property of the VM when passing to the limit in $\epsilon$.
\end{remark}

The main difference between the ESVM and the VM lies in the dynamics of the segregation. In fact, even for initially mixed densities, the ESVM system will lead at the limit to the relation $n_1^\infty n_2^\infty =0$. This is not the case for the VM, which does not enforce segregation at the limit.
However, taking initially segregated densities for the ESVM and for the VM, we get that $q^{m}(n_1^{\scriptsize{ini}}n_2^{\scriptsize{ini}})=0$ for the ESVM, and taking also $\alpha=0$ we find that both systems coincide. As a consequence their limits will also coincide in this case. The equivalence of the systems L-ESVM and L-VM when the initial densities $n_1^{\infty,ini}, n_2^{\infty,ini}$ are segregated can be obtained directly, taking furthermore $q^{\infty,ini}=q^{\infty}(t=0,\cdot)=0$ in the L-ESVM (Figure \ref{fig:link}). The transmission problem for the stationary L-VM is obtained by taking $q=0$ in $(T_2)$. Well-posedness and regularity for the stationary L-VM are then straightforward from Theorems \ref{wellpreg} and \ref{reg+}, when the assumptions (\ref{cc1})-(\ref{cc3}), and (\ref{cdd1})-\eqref{cc4} are still assumed.

\subsection{On the pressure jump in the stationary L-VM}

Finally, we show a striking feature: a pressure jump at the boundaries and interfaces in the stationary L-VM (or equivalently, the stationary L-ESVM when $q=0$). In the following we consider linear growth functions as in \eqref{lineargrowth} with $g_1,g_2,p_1^\ast,p_2^\ast>0$. Then using \eqref{relationpn} and \eqref{plim} we can express the pressure on $\Omega_1$, $\Omega_2$ and $\Omega^c$ as a function of the velocities $v_1$ and $v_2$ as: 

\begin{eqnarray}
p = \left\{
    \begin{array}{ll}
        p_1^{\ast}-\frac{1}{g_1}(\nabla \cdot v_1)_{\Omega_1}  & \mbox{ on } \Omega_1, \vspace{0.2cm}\\
        p_2^{\ast}-\frac{1}{g_2}(\nabla \cdot v_2)_{\Omega_2}  & \mbox{ on }  \Omega_2,\vspace{0.2cm}\\
        0  & \mbox{ on }  \Omega^c.
    \end{array}\label{expP}
\right.
\end{eqnarray}

\begin{proposition}[Pressure jump, stationary L-VM]\label{p1p2jump} Suppose (\ref{cc1})-(\ref{cc3}), (\ref{cdd1}) and \eqref{cc4}. Let $q=0$ and let $(v_1,v_2)\in H^{1}_0(\Theta)^{2}\times H^{1}_0(\Theta)^2$ be the solution of ($S_2$) given by Theorem \ref{wellpreg}. Then, the pressure $p$ of the L-VM given by \eqref{expP} is discontinuous across the interfaces $\Gamma_1 \cup \Gamma_2\cup \Gamma$. 
\end{proposition}

\begin{remark}
For the case of the L-ESVM, it is clear from $(T_{2,\Gamma_1})$, $(T_{2,\Gamma_2})$, and $(T_{2,\Gamma})$ that for the pressure to be continuous on the interfaces, it is required for $q$ to be identically zero on all interfaces $\Gamma_1 \cup \Gamma_2\cup \Gamma$. 
The pressure jump of the L-ESVM is then obvious for any non-vanishing $q$ at any point of the interfaces $\Gamma$, $\Gamma_1$ and $\Gamma_2$. Further work is needed in the general case.
\end{remark}

\section{Incompressible limit of the two species viscous model}\label{sec:proofincomp}

\subsection{Computation of the incompressible limit}

In this section we perform the proof of the formal limit of the ESVM defined by the system \eqref{n1new}-\eqref{pnew}.

\begin{proof}[Proof of Theorem \ref{formalim}] The incompressible limit of the system defined by the equations \eqref{n1new}-\eqref{pnew}, is obtained by taking $\epsilon \rightarrow{0}$, then  $m \rightarrow{\infty}$ and $\alpha \rightarrow{0}$.

First, the equation \eqref{pnew} can be written as,
$(1-n)p_{\epsilon}=\epsilon n$. Taking the limit as $\epsilon \rightarrow{0}$ we get,
\begin{equation}
    p^{\infty}(1-n^{\infty})=0.\label{pnl}
\end{equation}
Then, straightforward computations show that the expression of $q_m$ in Eq. \eqref{qmnew} gives,
\begin{equation*} \bigg(\frac{m-1}{m}q_m+1\bigg)^{\dfrac{m}{m-1}}=(1+r)^{m}=\bigg(\frac{m-1}{m}q_m+1\bigg)(1+r).
\end{equation*}
Passing to the limit when $m \rightarrow +\infty $, and using that we assume $q_m \xrightarrow{m\rightarrow{}\infty}q^\infty$, we obtain full segregation of the species at the limit,
\begin{equation}
    r^{\infty}=n_{1}^{\infty}n_{2}^{\infty}=0.\label{segq}
\end{equation}

Now, taking the limit in \eqref{n1new}--\eqref{p2new}, when $\epsilon \rightarrow 0, m\rightarrow + \infty$ and $\alpha \rightarrow 0$, under the assumption that all quantities converge, we get the system \eqref{nn1}-\eqref{pp1} for all $(t,x) \in [0; + \infty) \times \mathbb{R}^{d}$.

It now remains to obtain an equation to characterize the evolution of $p^\infty$ and $q^\infty$.

\paragraph{Equation for the congestion pressure.}

We obtain the equation satisfied by the total density $n=n_{1}+n_{2}$ by summing \eqref{n1new} and \eqref{n2new}:
\begin{equation}
    \partial_{t}n+ \nabla\cdot(n_{1}v_{1}+n_{2}v_{2}) +\alpha \nabla \cdot (n_{1} \nabla (\Delta n_{1})+n_{2} \nabla (\Delta n_{2}))=n_{1}G_{1}(p_{1})+n_{2}G_{2}(p_{2}).\label{n}
\end{equation}
To recover the equation that controls the pressure inside the domain $\Omega (t)$, we multiply \eqref{n} by $p'_{\epsilon}(n)=\frac{1}{\epsilon}(p_{\epsilon}+\epsilon)^2$, and obtain the following equation satisfied by $p_{\epsilon}$:
\begin{eqnarray}
       \epsilon \partial_{t}p_{\epsilon}+(p_{\epsilon}+\epsilon)^2\nabla \cdot (n_{1}v_{1}+n_{2}v_{2})&=&-(p_{\epsilon}+\epsilon)^2 \alpha \nabla \cdot(n_{1} \nabla (\Delta n_{1})+n_{2} \nabla (\Delta n_{2}))\nonumber \\
       &&+(p_{\epsilon}+\epsilon)^2(n_{1}G_{1}(p_{1})+n_{2}G_{2}(p_{2})), \nonumber
\end{eqnarray}
We pass to the limit $\epsilon \rightarrow {0}$ in the latter equation, then  $m \rightarrow {\infty}$, and $\alpha \rightarrow {0}$, we get at the limit Eq. \eqref{congpress}.
   
\paragraph{Equation for the repulsion pressure.}

Let $K_1=n_1q_m$, then the equation on $n_1 \partial_t q_m$ is as follows:
\begin{eqnarray}
n_1\partial_t q_m&+& r(q_m)'\nabla \cdot (n_1v_1)+n_1^2(q_m)'\nabla \cdot (n_2v_2) \nonumber \\
&=&n_1r(q_m)'(G_1(p_1)+G_2(p_2))
-\alpha r (q_m)'\nabla\cdot (n_{1} \nabla  (\Delta n_{1})) -\alpha n_1^2 (q_m)'\nabla\cdot (n_{2} \nabla  (\Delta n_{2})).\nonumber
\end{eqnarray}
Recalling that $q_m(r)=\frac{m}{m-1}[(1+r)^{m-1}-1],$ we compute $q'(r)=m(1+r)^{m-2}$, and we use the following expressions
$$(1+r)^{m-1}=\frac{m-1}{m}q_m+1 \quad \text{and \;}  \log(\frac{m-1}{m}q_m+1)=(m-1)\log(1+r).$$
Then we can write 
$$ r(q_m)'=\frac{r}{\log(1+r)}\frac{m}{m-1}(\frac{m-1}{m}q_m+1)^{\frac{m-2}{m-1}}\log(\frac{m-1}{m}q_m+1).$$
Replacing the expression of $r(q_m)'$ in the above equation we obtain:

\begin{eqnarray}
n_1\partial_t q_m&+&\nonumber \frac{r}{\log(1+r)}\frac{m}{m-1}(\frac{m-1}{m}q_m+1)^{\frac{m-2}{m-1}}\log(\frac{m-1}{m}q_m+1)\nabla \cdot (n_1v_1)\\\nonumber
&+&\frac{r}{\log(1+r)}\frac{m}{m-1}(\frac{m-1}{m}q_m+1)^{\frac{m-2}{m-1}}\log(\frac{m-1}{m}q_m+1)n_1\nabla \cdot v_2 \\\nonumber
&+&n_1^2v_2 \cdot\nabla n_2 (q_m)'\\\nonumber
&=& \frac{r}{\log(1+r)}\frac{m}{m-1}(\frac{m-1}{m}q_m+1)^{\frac{m-2}{m-1}}\log(\frac{m-1}{m}q_m+1)n_1(G_1+G_2)\nonumber\\
&&-\alpha \frac{r}{\log(1+r)}\frac{m}{m-1}(\frac{m-1}{m}q_m+1)^{\frac{m-2}{m-1}}\log(\frac{m-1}{m}q_m+1)  \nabla\cdot (n_{1} \nabla  (\Delta n_{1}))\nonumber \\
&&-\alpha n_1^2 (q_m)'\nabla\cdot (n_{2} \nabla  (\Delta n_{2})). \nonumber
\end{eqnarray}
The equation on $K_1$ becomes: 
\begin{eqnarray}
\partial_t K_1&+&\nonumber \frac{r}{\log(1+r)}\frac{m}{m-1}(\frac{m-1}{m}q_m+1)^{\frac{m-2}{m-1}}\log(\frac{m-1}{m}q_m+1)\nabla \cdot (n_1v_1)\\\nonumber
&+&\frac{r}{\log(1+r)}\frac{m}{m-1}(\frac{m-1}{m}q_m+1)^{\frac{m-2}{m-1}}\log(\frac{m-1}{m}q_m+1)n_1\nabla \cdot v_2 \\\nonumber
&+&n_1^2\nabla n_2 \cdot v_2 (q_m)'+q_m\nabla \cdot (n_1v_1)+q_m \alpha \nabla\cdot (n_{1} \nabla  (\Delta n_{1})) \\\nonumber
&=& \frac{r}{\log(1+r)}\frac{m}{m-1}(\frac{m-1}{m}q_m+1)^{\frac{m-2}{m-1}}\log(\frac{m-1}{m}q_m+1)n_1(G_1+G_2) \nonumber\\
&&-\alpha \frac{r}{\log(1+r)}\frac{m}{m-1}(\frac{m-1}{m}q_m+1)^{\frac{m-2}{m-1}}\log(\frac{m-1}{m}q_m+1)  \nabla\cdot (n_{1} \nabla  (\Delta n_{1}))\nonumber \\
&&-\alpha n_1^2 (q_m)'\nabla\cdot (n_{2} \nabla  (\Delta n_{2}))+q_mn_1G_1.\nonumber
\end{eqnarray}
Moreover we have the following equalities 
$$ n_1^2\nabla n_2 \cdot v_2 (q_m)'=n_1v_2\cdot \nabla q_m  - r(q_m)' \nabla n_1\cdot v_2,$$ $$ \alpha n_1^2 (q_m)'\nabla\cdot (n_{2} \nabla  (\Delta n_{2}))=\alpha [n_1 \nabla q_m -r(q_m)' \nabla n_1]\cdot \nabla \Delta n_2+\alpha n_1 r(q_m)' \Delta ^2 n_2.$$
Then the equation on $K_1$ becomes :
\begin{eqnarray}
\partial_t K_1&+&\nonumber \frac{r}{\log(1+r)}\frac{m}{m-1}(\frac{m-1}{m}q_m+1)^{\frac{m-2}{m-1}}\log(\frac{m-1}{m}q_m+1)\nabla \cdot (n_1v_1)\\\nonumber
&+&\frac{r}{\log(1+r)}\frac{m}{m-1}(\frac{m-1}{m}q_m+1)^{\frac{m-2}{m-1}}\log(\frac{m-1}{m}q_m+1)n_1\nabla \cdot v_2 \\\nonumber
&+&n_1 v_2 \cdot   \nabla q_m - \frac{r}{\log(1+r)}\frac{m}{m-1}(\frac{m-1}{m}q_m+1)^{\frac{m-2}{m-1}}\log(\frac{m-1}{m}q_m+1) \nabla n_1\cdot v_2  \nonumber \\
&&+q_m\nabla \cdot (n_1v_1)+q_m \alpha \nabla\cdot (n_{1} \nabla  (\Delta n_{1})) \nonumber\\
&=& \frac{r}{\log(1+r)}\frac{m}{m-1}(\frac{m-1}{m}q_m+1)^{\frac{m-2}{m-1}}\log(\frac{m-1}{m}q_m+1)n_1(G_1+G_2) \nonumber\\
&&-\alpha \frac{r}{\log(1+r)}\frac{m}{m-1}(\frac{m-1}{m}q_m+1)^{\frac{m-2}{m-1}}\log(\frac{m-1}{m}q_m+1)  \nabla\cdot (n_{1} \nabla  (\Delta n_{1}))\nonumber \\
&&-\alpha n_1 \nabla q_m \cdot \nabla \Delta n_2+\alpha \frac{r}{\log(1+r)}\frac{m}{m-1}(\frac{m-1}{m}q_m+1)^{\frac{m-2}{m-1}}\log(\frac{m-1}{m}q_m+1) \nabla n_1 \cdot \nabla \Delta n_2\nonumber \\
&&-\alpha n_1 \frac{r}{\log(1+r)}\frac{m}{m-1}(\frac{m-1}{m}q_m+1)^{\frac{m-2}{m-1}}\log(\frac{m-1}{m}q_m+1) \Delta ^2 n_2+q_m n_1G_1.\nonumber
\end{eqnarray}
Finally, assuming that the limit to all the quantities exists, and remembering \eqref{segq}, we formally pass to the limit in the above equation when $m$ goes to infinity and $\alpha, \epsilon$ go to zero and obtain on $\Theta$,

\begin{eqnarray}
\partial_t K_1^{\infty}&+&\nonumber (q^{\infty}+1)\log(q^{\infty}+1)\nabla \cdot (n_1^{\infty}v_1^{\infty})+(q^{\infty}+1)\log(q^{\infty}+1)n_1^{\infty}\nabla \cdot v_2^{\infty} \\\nonumber
&+&n_1^{\infty} v_2^{\infty} \cdot   \nabla q^{\infty} - (q^{\infty}+1)\log(q^{\infty}+1)\nabla n_1^{\infty}\cdot v_2^{\infty} +q^{\infty}\nabla \cdot (n_1^{\infty}v_1^{\infty}) \nonumber\\
&=& (q^{\infty}+1)\log(q^{\infty}+1)n_1^{\infty}(G_1(p_1^{\infty})+G_2(p_2^{\infty}))+q^{\infty}n_1^{\infty}G_1(p_1^{\infty}).\nonumber 
\end{eqnarray}
An analogous computation holds for $K_2^{\infty}$ on $\Theta$, leading to Eq. \eqref{u2inf}. This ends the proof of the theorem.
\end{proof}

\subsection{Free boundary problem}

We now explain how this system generates a free boundary problem. Motivated by the relation \eqref{pnl}, we define the domain,
\begin{eqnarray}
\Omega(t)\coloneqq \{x | \;p^{\infty}( \cdot, t)>0\}, \label{defOmega}
\end{eqnarray}
and we note that
\begin{eqnarray}
\Omega(t)=\{x | \;p^{\infty}( \cdot, t)>0\} \subset \{x |\;n^{\infty}( \cdot, t)=1\}, \; a.e. \label{inclus}
\end{eqnarray}

Using Eq. \eqref{segq} and the inclusion \eqref{inclus}, we can decompose $\Omega(t)$ into two subdomains $\Omega_1(t)$ and $\Omega_2(t)$ such that
\begin{equation}\Omega_{1}(t)=\{x |\; n_{1}^{\infty}=1\}\cap \Omega(t)\; \text{ and } \; \Omega_{2}(t)=\{x |\; n_{2}^{\infty}=1\} \cap \Omega(t), \nonumber
\end{equation}
with
\begin{equation}\Omega_{1}(t) \cap \Omega_{2}(t)=\varnothing \; \text{and} \; \Omega_{1}(t) \cup \Omega_{2}(t)=\Omega(t). \nonumber
\end{equation}

In fact, in (\ref{inclus}), the two domains coincide almost everywhere, for almost all times. Indeed, let us assume that on some time-space domain $U\in \mathbb{R}^{d+1}$, we have simultaneously $p^\infty=0$ and $n^\infty=1$ (with for example $n_1^\infty=1$ and $n_2^\infty=0$).
 Then by equation (\ref{nn1}) we have $0=\partial_t n_1^\infty=n_1^\infty(G_1(0)-\nabla  \cdot v_1^\infty )$ in $U$, so that $\nabla \cdot v_1=G_1(0)=g_1p_1^\ast >0$ in $U$. On the other hand, applying the divergence operator to (\ref{vv1}), one gets $\nabla \cdot v_1=0,$ which is absurd. Therefore, $U$ is of measure zero, and finally,
\begin{eqnarray}
\Omega(t)=\{x | \;p^{\infty}( \cdot, t)>0\} = \{x |\;n^{\infty}( \cdot, t)=1\}, \; \text{a.e, a.e }t\ge0. 
\nonumber
\end{eqnarray}
Furthermore, the complementary relations \eqref{congpress}, \eqref{u1inf} and \eqref{u2inf} can be rewritten on each subdomain. First, Eq. (\ref{congpress}) does not give additional information outside $\Omega(t)$, but on $\Omega(t)$ it can be rewritten as,
 \begin{eqnarray}
        \nabla \cdot v_{1}^{\infty}= G_{1}(p_{1}^{\infty}) \quad \text { on } \Omega_1 (t), \qquad \text{and} \qquad
        \nabla \cdot v_{2}^{\infty}= G_{2}(p_{2}^{\infty}) \quad \text { on } \Omega_2 (t). \nonumber
        \end{eqnarray}

Then, from \eqref{u1inf}, the equation of $q^{\infty}$ on $\Omega_1(t)$ becomes,
\begin{eqnarray}
\partial_t q^{\infty}+ (q^{\infty}+1)\nabla \cdot (\log(q^{\infty}+1)v_2^{\infty})=(q^{\infty}+1)\log(q^{\infty}+1)G_2(p_2^{\infty}). \label{q1inf}
\end{eqnarray}
Similarly in $\Omega_2(t)$ we obtain from \eqref{u2inf},
\begin{eqnarray}
\partial_t q^{\infty}+ (q^{\infty}+1)\nabla \cdot (\log(q^{\infty}+1)v_1^{\infty})=(q^{\infty}+1)\log(q^{\infty}+1)G_1(p_1^{\infty}).\label{q2inf}
\end{eqnarray}
Finally, we compute $K_1^{\infty}$ and $K_2^{\infty}$ in $\Theta \backslash \Omega(t)$ where $0\leq n_1^{\infty}<1$ and $ 0 \leq n_2^{\infty}<1$.  Taking $p_1^{\infty}=0$ in \eqref{u1inf} in areas where $n_2^{\infty}=0$ and $0<n_1^{\infty}<1$ we derive the following equation,
\begin{eqnarray}
\partial_t K_1^{\infty}&+&\nonumber (q^{\infty}+1)\log(q^{\infty}+1)\nabla \cdot (n_1^{\infty}v_1^{\infty})+(q^{\infty}+1)\log(q^{\infty}+1)n_1^{\infty}\nabla \cdot v_2^{\infty} \\\nonumber
&+&n_1^{\infty} v_2^{\infty} \cdot   \nabla q^{\infty} - (q^{\infty}+1)\log(q^{\infty}+1)\nabla n_1^{\infty}\cdot v_2^{\infty} +q^{\infty}\nabla \cdot (n_1^{\infty}v_1^{\infty}) \nonumber\\
&=& (q^{\infty}+1)\log(q^{\infty}+1)n_1^{\infty}(G_1(0)+G_2(p_2^{\infty}))+q^{\infty}n_1^{\infty}G_1(0),
\nonumber
\end{eqnarray}
And to compute $K_2^{\infty}$ in areas where $n_1^{\infty}=0$ and $0<n_2^{\infty}<1$, we take $p_2^{\infty}=0$ in \eqref{u2inf} and derive the following equation: 

\begin{eqnarray}
\partial_t K_2^{\infty}&+&\nonumber (q^{\infty}+1)\log(q^{\infty}+1)\nabla \cdot (n_2^{\infty}v_2^{\infty})+(q^{\infty}+1)\log(q^{\infty}+1)n_2^{\infty}\nabla \cdot v_1^{\infty} \\\nonumber
&+&n_2^{\infty} v_1^{\infty} \cdot   \nabla q^{\infty} - (q^{\infty}+1)\log(q^{\infty}+1)\nabla n_2^{\infty}\cdot v_1^{\infty} +q^{\infty}\nabla \cdot (n_2^{\infty}v_2^{\infty}) \nonumber\\
&=& (q^{\infty}+1)\log(q^{\infty}+1)n_2^{\infty}(G_1(p_1^{\infty})+G_2(0))+q^{\infty}n_2^{\infty}G_2(0).
\nonumber
\end{eqnarray}

Finally, note that the equations \eqref{u1inf} and \eqref{u2inf} are very related to the choice of the repulsion pressure $q_m$. In particular, the presence of the terms $\log(q^\infty+1)$ comes from the power law in Eq. \eqref{qmnew}.

\subsection{Velocities of the interfaces}

Finally, we complete here the derivation of the free boundary problem entailed by the incompressible limit by computing the velocities of the interfaces. To this end, we consider the case where each tissue $i$ occupies fully its domain $\Omega_i$, that is, we assume that $n_{1}^{\infty}=\chi_{\Omega_1}(t)$ and $n_{2}^{\infty}=\chi_{\Omega_1}(t)$. One can verify that in this case $(n_1^\infty,n_2^\infty)$ solve the equations \eqref{nn1} and \eqref{nn2} on each subdomain $\Omega_i(t)$, $i=1,\,2$ and $\Omega^c(t)$.

 We then compute the velocity of $\partial \Omega(t)=(\partial \Omega_{1}(t)\cap \partial \Omega(t)) \cup (\partial \Omega_{2}(t)\cap \partial \Omega(t))$ and of the interface between the two tissues $\Gamma(t)=\partial \Omega_{1}(t)\cap \partial \Omega_{2}(t)$. 

 We test \eqref{nn1} with some $\phi \in \mathcal{C}^{\infty}_{c}(\mathbb{R}^d)$,
 \begin{eqnarray}
  \partial_{t} \int_{\R^{d}}n_{1}^{\infty}\phi
  = \int_{\R^{d}}\partial_{t}n_{1}^{\infty}\phi
  \nonumber
  =\int_{\R^{d}} n_{1}^{\infty}v_{1}^{\infty}\cdot\nabla \phi+ \int_{\R^{d}}n_{1}^{\infty}G_{1}(p_{1}^{\infty})\phi.
  \nonumber
 \end{eqnarray}
 Hence using Green's formula, and that $n_{1}^{\infty}=0$ on $\Omega_{1}^{c}(t),$ we can write,
  \begin{eqnarray}
        \partial_{t} \int_{\Omega_{1}(t)}\phi=
        -\int_{\Omega_{1}(t)} \nabla \cdot v_{1}^{\infty} \phi+\int_{\partial\Omega_{1}(t)}v_{1}^{\infty}\cdot\vec{\nu}\phi 
        + \int_{\Omega_{1}(t)}G_{1}(p_{1}^{\infty})\phi.
        \nonumber
 \end{eqnarray}
 with $\vec{\nu}$ the outward normal vector to $\Omega_{1}(t)$.
Using the complementary relation \eqref{congpress} on the subdomain $\Omega_{1}(t)$ we obtain
\begin{eqnarray}
        \partial_{t} \int_{\Omega_{1}(t)}\phi&=&
        \int_{\partial\Omega_{1}(t)\cap \partial\Omega(t)}v_{1}^{\infty}\cdot\vec{\nu} \phi +\int_{\Gamma(t)}v_{1}^{\infty}\cdot\vec{\nu}\phi.\label{veln}
 \end{eqnarray}
recalling that $\Gamma(t)=\partial\Omega_{1}(t)\cap\partial \Omega_{2}(t)$ is the interface between the two densities, and $\partial \Omega_{1}(t)\cap \partial \Omega(t)$ the exterior boundary of $\Omega_{1}(t)$.

We now introduce $V_{\partial \Omega_{1}(t)\cap \partial \Omega(t)}$ the velocity of the exterior boundary of $\Omega_{1}(t)$ along the outward normal vector $\vec{\nu}$, and $V_{\Gamma(t)}$ the velocity of the interface $\Gamma(t)$ along the vector $\vec{\nu}$.
Using Reynolds transport theorem on moving domains, the LHS gives
 \begin{eqnarray}
       \partial_{t}  \int_{\Omega_{1}(t)} \phi&=&\int_{\partial \Omega_{1}(t)}V_{\partial \Omega_{1}(t)}\cdot  \vec{\nu}\phi \nonumber \\
       &=& \int_{\partial\Omega_{1}(t)\cap \partial\Omega(t)}V_{\partial \Omega_{1}(t)}\cdot  \vec{\nu}\phi +\int_{\Gamma(t)}V_{\partial \Omega_{1}(t)}\cdot  \vec{\nu}\phi \label{leib1}.
 \end{eqnarray}
Since \eqref{veln} and \eqref{leib1} are satisfied for all $\phi \in \mathcal{C}^{\infty}_{c}(\R^{d})$, particularly by extension of $\phi \in \mathcal{C}^{\infty}_{c}(\partial \Omega_{1}(t)\cap \partial \Omega(t))$ we deduce the following relation,
\begin{equation}
    V_{\partial \Omega_{1}(t)\cap \partial \Omega(t)}=v_{1}^{\infty}\cdot\vec{\nu}, \nonumber
\end{equation}
 similarly by extension of $\phi \in \mathcal{C}^{\infty}_{c}(\Gamma(t))$ we deduce,
 \begin{equation}
    V_{\Gamma(t)}=v_{1}^{\infty}\cdot\vec{\nu}. \nonumber
\end{equation}
Finally, by a symmetric reasoning, defining $\vec{\mu}$ the outward normal vector to $\Omega_{2}(t)$, we obtain the velocity of the exterior boundary of $\Omega_{2}(t)$,
\begin{equation}
    V_{\partial \Omega_{2}(t)\cap \partial \Omega(t)}=v_{2}^{\infty}\cdot\vec{\mu}, \nonumber
\end{equation}
as well as the continuity of the velocity along the interface $\Gamma(t)$,
 \begin{equation}
   v_{1}^{\infty}\cdot\vec{\nu}= v_{2}^{\infty}\cdot\vec{\nu}\quad \text{on\;} \Gamma(t). \nonumber
\end{equation}

\subsection{On the limit of the repulsion pressure} \label{section2.3}

In this section we study the behavior of the repulsion pressure $q_m$ in the mechanical model at the incompressible limit. We notice that $q^\infty=0$ gives an admissible solution for the problem \eqref{nn1}--\eqref{u2inf}. Indeed, since the system at the incompressible limit is fully segregated \eqref{segr}, it might be relevant to consider that at the limit the repulsion pressure vanishes. The question we want to address is the following: is it relevant to consider cases where the repulsion pressure persists at the incompressible limit, that is, $q^\infty$ is not identically null ? Our numerical simulations suggest that $q^\infty$ does not vanish at the incompressible limit, and they highlight its role in creating adjacent zones of opposite curls in the anterior part of the PSM. To explore further the role of this repulsion force, we proceed with the following formal analysis.

Here we are interested in the case where densities at the incompressible limit are indicator functions (which is a case where the free boundary model is particularly relevant). We then suppose that $n_1^m$ and $n_2^m$ converge respectively towards $\chi_{\Omega_1}$ and $\chi_{\Omega_2}$ in a sufficiently strong sense, locally uniformly in time, and we write the Taylor expansion of $n_1^m$ and $n_2^m$ with respect to the parameter $m$ (assuming that such expansion exists),
\begin{eqnarray} \label{eq:DLn1}
n_1^m = \chi_{\Omega_1} + \frac1m h_1 + o (m^{-1}),\\ \label{eq:DLn2}
n_2^m = \chi_{\Omega_2} + \frac1m h_2 + o(m^{-1}).
\end{eqnarray}
We can then compute the product,
\begin{eqnarray}
n_1^m n_2^m = \frac1m \left( h_2 \chi_{\Omega_1} +  h_1 \chi_{\Omega_2} \right) + o (m^{-1}), \nonumber
\end{eqnarray}
so that the repulsion pressure becomes at the limit,
\begin{eqnarray}
q_m &=&\frac{m}{m-1}[(1+n_1^mn_2^m)^{m-1}-1]\nonumber \\
&=&\frac{m}{m-1}\left(\exp\left\{(m-1)\log\left(1+ \frac1m ( h_2 \chi_{\Omega_1} +  h_1 \chi_{\Omega_2} ) + o(m^{-1}) \right)\right\}-1\right)\nonumber\\
&\xrightarrow{m\rightarrow +\infty}& \exp\left\{ h_2 \chi_{\Omega_1} +  h_1 \chi_{\Omega_2} \right\}-1. \label{lim:DLqm}
\end{eqnarray}
Finally, rewriting this last expression, we have shown that,
\begin{eqnarray}
q^\infty = \chi_{\Omega_1} \left(\exp( h_2 )-1\right) + \chi_{\Omega_2} \left(\exp( h_1 )-1\right). \label{eq:DLqm}
\end{eqnarray}

\paragraph{Interpretation.} In the Taylor expansion \eqref{eq:DLn1} expressed in terms of small values of the parameter $1/m$, the term of order zero gives the asymptotic behavior of the density $n_1^m$. We see that the coefficient of first order in $1/m$ in \eqref{eq:DLn1} (i.e the quantity $h_1$) appears in the final expression of $q^\infty$ at the incompressible limit under the form $h_1 \chi_{\Omega_2}$. It is remarkable that the repulsion force persists \emph{a priori} everywhere on the domain, and not only at the interface. One could interpret this phenomenon as a repulsion force emerging from microscopic residuals of the tissue $1$ inside tissue $2$ (and \emph{vice versa}). The situation where the repulsion force vanishes at the limit then corresponds to a modelling situation where such microscopic effects are neglected (fast segregation regime). Note that $q^\infty=0$ gives a particular solution of \eqref{q1inf}--\eqref{q2inf} (see remark \ref{qzero}).

In fact, $q^\infty$ cannot be neglected in general at the limit as it produces an effect on the dynamics. We note that in \eqref{eq:DLqm} the repulsion only persists in the domain $\Omega_1 \cup \Omega_2$ (see Remark \ref{rem:q_out_Omega}), that is, the repulsion force produces a finite effect only in regions fully occupied by the densities.

Moreover, the persistence of the repulsion pressure $q^\infty$ at the limit can be seen as a \emph{ghost effect}. This terminology originates from the framework of rarefied gas dynamics. For instance, in rarefied gases, steady flows can be induced by temperature fields. When a continuum limit is performed on the Boltzmann equation on the basis of kinetic theory, in the sense that the Knudsen number of the system tends to zero, these flows vanish. This means that no condensation nor evaporation occurs, and the components are at rest. However, it was shown that these vanishing flows produce a finite effect on the gas behavior at the limit. This is known as the \emph{ghost effect}, originally discovered by Sone et al \citep{Sone1, Sone2, Sone4, Sone3, Sone5}. The \emph{ghost effect} has been described in different physical settings (\citep{Sone6, Sone7, Sone8, Sone1}. These works show that the Navier-Stokes equation (or the heat-conduction equation) must be coupled with the \emph{ghost effect} to fully describe the gas dynamics at the continuum limit \citep{Sone4}. This was shown by doing an asymptotic analysis on the Boltzmann equation, by expansion in the Knudsen number, which allowed the derivation of the fluid-dynamics equations at the limit and the \emph{ghost effect}. In our case, if we assume that the Taylor expansion in $1/m$ holds, the \emph{ghost effect} appears in terms of first order in $1/m$ similarly to \citep{Sone6, Sone4, Sone5}.

\section{Analysis of the stationary L-ESVM ($q$ given)}\label{sec:tran_mod1}

In this section, we study the stationary L-ESVM in the framework described in section \ref{sec:studystatLESVM}. We recall that the system is considered in a bounded domain $\Theta \subset \mathbb{R}^{2}$ ($d=2$) with homogeneous Dirichlet boundary conditions on the velocity. We consider three subdomains $\Omega,\,\Omega_1$ and $\Omega_2$ which
satisfy the assumptions \eqref{cc1}--\eqref{cc3}.
In the simplified framework considered here, the equations \eqref{u1inf} and \eqref{u2inf} are removed and instead, $q$ is given as in \eqref{qL2}. We also choose the growth functions $G_1$ and $G_2$ to be linear as in \eqref{lineargrowth}. For clarity, we remove the superscripts $"\infty"$ from all the variables in this section.

The resulting problem is the following system on $(v_1,v_2)$,
$$
(S_2) \;\; 
\left\{
\begin{array}{ll}
-\beta_{1}\Delta v_{1}+v_{1}=- \nabla [(p_1^{\ast}-\frac{1}{g_1}\nabla \cdot v_1)\chi_{\Omega_1}+(p_2^{\ast}-\frac{1}{g_2}\nabla \cdot v_2+q)\chi_{\Omega_2}]& \text{\; on \;} \Theta,  \\
-\beta_{2}\Delta v_{2}+v_{2}=- \nabla [(p_2^{\ast}-\frac{1}{g_2}\nabla \cdot v_2)\chi_{\Omega_2}+(p_1^{\ast}-\frac{1}{g_1}\nabla \cdot v_1+q)\chi_{\Omega_1}]& \text{\; on \;} \Theta, \\
    v_1=v_2=0& \text{\; on \;} \partial \Theta, 
    \end{array}
     \right.
$$
with $\chi_{\Omega_1},\chi_{\Omega_2}$ the indicator functions of the domains $\Omega_1,\Omega_2$ respectively, and where $g_1,\,g_2,\,p_1^\ast,\, p_2^\ast>0$ and $q$ is a given function on $\Omega$.

\subsection{Well-posedness and regularity of the stationary linear system.}

We prove here Theorem \ref{wellpreg}. We first start with the well-posedness.

\begin{proposition}\label{prop:LaxMilgram2} For $q\in L^2(\Omega)$, there exists a unique solution $(v_1,v_2) \in H^{1}_0(\Theta)^{2}\times H^{1}_0(\Theta)^2$ to problem $(S_2)$ under the condition \eqref{cdd1} on the model parameters.
\end{proposition}
\begin{proof} The proof is done using the Lax-Milgram theorem. We can write the weak formulation of $(S_2)$ as
\begin{equation}
 \;\; 
 (\tilde{V})
\left\{
\begin{array}{ll}
\text{Find $(v_1,v_2) \in H^{1}_0(\Theta)^{2}\times H^1_0(\Theta)^2$, such that }\\
       B((v_1,v_2),(\phi_1,\phi_2))=l((\phi_1,\phi_2)), \quad \forall (\phi_1,\phi_2) \in H^{1}_{0}(\Theta)^{2}\times H^{1}_{0}(\Theta)^{2},\nonumber
    \end{array}
     \right.
\end{equation}
where we define the bilinear map $B$ as
\begin{eqnarray}
B((v_1,v_2),(\phi_1,\phi_2))&=&\beta_1 \int_{\Theta}\nabla v_1 : \nabla \phi_1+\beta_2 \int_{\Theta}\nabla v_2 : \nabla \phi_2 +\int_{\Theta} v_1\cdot \phi_1+\int_{\Theta} v_2\cdot \phi_2 \nonumber \\
&&+\frac{1}{g_1} \int_{\Theta} (\nabla \cdot v_1 )(\nabla \cdot \phi_1 ) \chi_{\Omega_1} +\frac{1}{g_2} \int_{\Theta} (\nabla \cdot v_2 )(\nabla \cdot \phi_2 ) \chi_{\Omega_2}\nonumber \\
&&+\frac{1}{g_1} \int_{\Theta} (\nabla \cdot v_1 )(\nabla \cdot \phi_2 ) \chi_{\Omega_1}+\frac{1}{g_2} \int_{\Theta} (\nabla \cdot v_2 )(\nabla \cdot \phi_1 ) \chi_{\Omega_2}, \nonumber
\end{eqnarray}
and the linear application $l$ as
\begin{eqnarray}
l((\phi_1,\phi_2))&=& \int_{\Omega_1}(p_1^{\ast}+q)(\nabla \cdot \phi_1+\nabla \cdot \phi_2)  + \int_{\Omega_2}(p_2^{\ast}+q)(\nabla \cdot \phi_2+\nabla \cdot \phi_1). \nonumber
\end{eqnarray}
To prove the continuity of the bilinear application we do the following: 
\begin{eqnarray}
    |B((v_1,v_2),(\phi_1,\phi_2))| & \leq& \beta_1 \|\nabla v_1\|_{L^{2}(\Theta)} \|\nabla \phi_1\|_{L^{2}(\Theta)} + \|v_1\|_{L^{2}(\Theta)}\|\phi_1\|_{L^{2}(\Theta)}\nonumber \\
    &&+ \frac{1}{g_1} \|\nabla \cdot v_1\|_{L^{2}(\Theta)} \|\nabla \cdot \phi_1\|_{L^{2}(\Theta)}+\beta_2 \|\nabla v_2\|_{L^{2}(\Theta)} \|\nabla \phi_2\|_{L^{2}(\Theta)} \nonumber \\
    &&+ \|v_2\|_{L^{2}(\Theta)}\|\phi_2\|_{L^{2}(\Theta)}+ \frac{1}{g_2} \|\nabla \cdot v_2\|_{L^{2}(\Theta)} \|\nabla \cdot \phi_2\|_{L^{2}(\Theta)}\nonumber \\
    &&+\frac{1}{g_1} \|\nabla \cdot v_1\|_{L^{2}(\Theta)} \|\nabla \cdot \phi_2\|_{L^{2}(\Theta)}+\frac{1}{g_2} \|\nabla \cdot v_2\|_{L^{2}(\Theta)} \|\nabla \cdot \phi_1\|_{L^{2}(\Theta)} \nonumber \\
     & \leq& (1+\min (\beta_1 +\frac{2}{g_1},\beta_2+\frac{2}{g_2}))\|(v_1,v_2)\|_{H^{1}_{0}(\Theta)^{2}\times H^{1}_{0}(\Theta)^{2}} \|(\phi_1,\phi_2)\|_{H^{1}_{0}(\Theta)^{2}\times H^{1}_{0}(\Theta)^{2}},\nonumber
\end{eqnarray}
where we used Cauchy-Schwarz inequality and that
\begin{equation}\|\nabla \cdot v \|^{2}_{L^{2}(\Theta)} \leq \|\nabla v \|^{2}_{L^{2}(\Theta)}.\label{divandgrad}
\end{equation}
For the coercivity of the bilinear application we use again Cauchy-Schwarz's inequality and the inequality \eqref{divandgrad},
\begin{eqnarray}
    B((v_1,v_2),(v_1,v_2)) &=& \beta_1 \|\nabla v_1\|^2_{L^{2}(\Theta)} + \|v_1\|^2_{L^{2}(\Theta)}+ \frac{1}{g_1} \|\nabla \cdot v_1\|^2_{L^{2}(\Omega_1)}
    \nonumber \\
    && +\beta_2 \|\nabla v_2\|^2_{L^{2}(\Theta)} + \|v_2\|^2_{L^{2}(\Theta)}+ \frac{1}{g_2} \|\nabla \cdot v_2\|^2_{L^{2}(\Omega_2)}
    \nonumber \\
    & &+ \frac{1}{g_1} \int_{\Omega_1}\nabla \cdot v_1 \nabla \cdot v_2
    +\frac{1}{g_2} \int_{\Omega_2}\nabla \cdot v_2\nabla \cdot v_1   
    \nonumber \\
    &\geq& \beta_1 \|\nabla v_1\|^2_{L^{2}(\Theta)}
    + \|v_1\|^2_{L^{2}(\Theta)}
    +\beta_2 \|\nabla v_2\|^2_{L^{2}(\Theta)} + \|v_2\|^2_{L^{2}(\Theta)}
    \nonumber \\
    & &- \frac{1}{4 g_1} \|\nabla \cdot v_2\|^2_{L^{2}(\Omega_1)}
    - \frac{1}{4 g_2} \|\nabla \cdot v_2\|^2_{L^{2}(\Omega_1)} \nonumber \\
    &\geq& \left( \beta_1 - \frac{1}{4 g_2} \right) \|\nabla v_1\|^2_{L^{2}(\Theta)}
    + \|v_1\|^2_{L^{2}(\Theta)}
    +\left( \beta_2 - \frac{1}{4 g_1} \right) \|\nabla v_2\|^2_{L^{2}(\Theta)} + \|v_2\|^2_{L^{2}(\Theta)}, \nonumber
\end{eqnarray}
which gives the coercivity of $B$ under the condition (\ref{cdd1}).

Finally we obtain the continuity of the linear application $l$ by :
\begin{eqnarray}
   |l((\phi_1,\phi_2))| & \leq & C (p_1^{\ast}+p_2^{\ast}+\|q\|_{L^{2}(\Omega)}) \|(\phi_1,\phi_2)\|_{H^{1}_{0}(\Theta)^{2}\times H^{1}_{0}(\Theta)^{2}}. \nonumber
\end{eqnarray}
Then by the Lax-Milgram theorem there exists a unique solution $(v_1,v_2) \in H^{1}_{0}(\Theta)^{2}\times H^{1}_{0}(\Theta)^{2}$ to $(S_2).$
\end{proof}

Whenever the repulsion pressure $q$ is assumed to be smooth on each subdomain, we can obtain more regularity on the velocities by applying the results of elliptic regularity for transmission problems from \cite{li}:
\begin{proposition}[Elliptic regularity]\label{prop:ellip} For a given $q\in \mathcal{C}^{0,\mu}(\overline{\Omega_1})\cap \mathcal{C}^{0,\mu}(\overline{\Omega_2})$ for some $0<\mu<1$, and if the domains satisfy the conditions (\ref{cc1})-(\ref{cc3}) and the condition on the model parameters \eqref{cdd1}, then the solution $(v_1,v_2)$ to the system $(S_2)$ is $\mathcal{C}^{1,\alpha'}(\overline{\Omega_1})\cap\mathcal{C}^{1,\alpha'}( \overline{\Omega_2}) \cap \mathcal{C}^{1,\alpha'}( \overline{\Omega^c})$, with $0<\alpha'\leq\min(\mu,\frac{1}{4})$.
\end{proposition}

\begin{proof} This is a direct application of Theorem 1.1 and Remark 1.2 in \cite{li}. To apply this theorem, we first check that the domains satisfy the appropriate requirements. In fact, we assume \eqref{cc1b} and \eqref{cc3}.
Then, it remains to check that the second order differential operator of $(S_2)$ satisfies the weak ellipticity condition, that is
\begin{eqnarray}
 B((\phi^1,\phi^2),(\phi^1,\phi^2))\geq \lambda \|\nabla \phi\|^2_{L^2(\Theta)},\qquad  \forall \phi=(\phi^1,\phi^2) \in H^1_0(\Theta)^4, \nonumber
\end{eqnarray}
for some $\lambda>0$.
This condition coincides with the coercivity condition of the bilinear application $B$. It is satisfied with $\lambda=\min\left(\beta_1-\frac{1}{4 g_2},\beta_2-\frac{1}{4 g_1}\right)$, which is positive thanks to \eqref{cdd1}.

Finally, note that Theorem 1.1 in \cite{li} gives the regularity result away from $\partial\Omega$. Thanks to the assumption in \eqref{cc2}, this already gives the regularity on $\overline{\Omega_i}$ for $i=1,\, 2$. Thanks to the homogeneous Dirichlet boundary conditions, to obtain the regularity up to $\partial \Theta$, it suffices to slightly move away the boundary $\partial \Theta$, extending the velocities by zero on the new domain.
\end{proof}


\begin{proof}[Proof of Theorem \ref{wellpreg}] Finally, Theorem \ref{wellpreg} is a consequence of Propositions \ref{prop:LaxMilgram2} and \ref{prop:ellip}.
\end{proof}

\subsection{Transmission problem of the stationary linear system}
We show here that System $(S_2)$ can be rewritten as the transmission problem $(T_2)$.

\begin{proof}[Proof of Proposition \ref{prop:TP}] Recall that the conditions (\ref{cc1})-(\ref{cc3}) and (\ref{cdd1}) are still assumed and that $q\in \mathcal{C}^{0,\mu}(\overline{\Omega_1})\cap \mathcal{C}^{0,\mu}(\overline{\Omega_2})$, for some $\mu\in(0,1)$.

First, by testing System $(S_2)$ with a test function compactly supported in $\Omega_1$, it is clear that the system $(T_{2,\Omega_1})$ is satisfied in a distributional sense (actually, against test functions in $H^1_0(\Omega_1)$ thanks to the regularity obtained from Theorem \ref{wellpreg}, and even in a classical sense thanks to interior elliptic regularity). The same applies in the other subdomains, so that $(T_{2,\Omega_1})$, $(T_{2,\Omega_2})$ and $(T_{2,\Omega^c})$ are satisfied on their respective domains.

To obtain the transmission conditions, we consider a test function $\phi \in \mathcal{C}^\infty_c(\Theta)^2$, and thanks to $\chi_{\Omega_1^n},\, \chi_{\Omega_2^n},$ and $ \chi_{\Omega_c^n}$ smooth approximations of respectively $\chi_{\Omega_1},\,\chi_{\Omega_2},$ and $\chi_{\Omega^c}$, we construct a new test function $\tilde{\phi}_{n} \in \mathcal{C}^\infty_c(\Theta)^2$ as
$\tilde{\phi_n}=\phi[1-\chi_{\Omega_1^n}-\chi_{\Omega_2^n}-\chi_{\Omega_c^n}]$.
Then we multiply (dot product) each one of the equations of $(S_2)$ with the test function $\tilde{\phi}_{n}$ and integrate, and we pass to the limit when $n$ tends to infinity. By identifying the quantities thus obtained on each section of the boundaries, one gets the transmission conditions $(T_{2,\Gamma_1})$, $(T_{2,\Gamma_2})$ and $(T_{2,\Gamma})$ (see Section \ref{sec:appendix} for details on the derivation of the transmission problem in the simple case of one species which can be similarly extended to the case of two species).
\end{proof}

\subsection{Further regularity for the stationary linear system}

In this section we show the proof of Theorem \ref{reg+}. We recall that the conditions (\ref{cc1})-(\ref{cc3}) and (\ref{cdd1})-(\ref{cc4}) are assumed and that $q\in \mathcal{C}^{\infty}(\overline{\Omega_1})\cap \mathcal{C}^{\infty}(\overline{\Omega_2})$.
\begin{proof}[Proof of Theorem \ref{reg+}] The regularity result is a direct consequence of Proposition 1.4 in \cite{li}. In fact, the well-posedness of the stationary L-ESVM for a given $q\in \mathcal{C}^{\infty}(\overline{\Omega_1})\cap \mathcal{C}^{\infty}(\overline{\Omega_2})$ is ensured by Theorem \ref{wellpreg}. Moreover, under (\ref{cc4}), we do not allow more than two subdomains to be in contact, thus all the boundary domains $\partial\Omega_1, \partial\Omega_2,$ and $\partial\Omega$ are smooth, and we gain the desired regularity.
\end{proof}

\section{Analysis of the L-VM}\label{sec:tran_pres}
 In this section we first show a formal proof of the segregation property of the VM.  We are then interested in the quantitative behavior of the stationary L-VM. We show the existence of a pressure jump using the formalism of a transmission problem.
\subsection{Proof of the segregation property}
This section is dedicated to the formal proof of Proposition \ref{segprop}. It is based on the study of the evolution of the population fraction.

\begin{proof}[Proof of Proposition \ref{segprop}]
We introduce the population fraction, on the set where $n>0$,
\begin{equation*}
c\coloneqq  \frac{n_{1}}{n}.
\end{equation*}
We note that the full segregation \eqref{fullsegeps} can be written as,
\begin{equation*}
c(1-c)=0.
\end{equation*}

We use equations \eqref{n1v} and \eqref{n2v} to derive an equation for $c$,
\begin{eqnarray*}
\partial_t c+\frac{(1-c)}{n}\nabla n_{1}\cdot v_{1}-\frac{c}{n}\nabla n_{2}\cdot v_{2}+c(1-c)(\nabla \cdot v_{1}-\nabla \cdot v_{2})\\
=c(1-c)(G_1(p_\epsilon)-G_2(p_\epsilon)).
\end{eqnarray*}
Then we can see that if $c^{\scriptsize{ini}}(1-c^{\scriptsize{ini}})=0$, then we get
\begin{equation*}
c(1-c)=0, \qquad \forall (t,x) \in [0;+\infty)\times \mathbb{R}^d .\label{sisi} 
\end{equation*}
\end{proof}

\subsection{Pressure jump of the stationary L-VM}
The corresponding stationary transmission problem for the L-VM is a special case of the transmission problem for the L-ESVM in the case where $q=0$. From proposition \ref{prop:TP}, we obtain the following transmission problem by taking $q=0$ in $(T_2)$,
\begin{equation*}
(T_{VM})
\left\{
\begin{array}{ll}
-\beta_1 \Delta v_1+v_1=0 & \text{\; in \;} \Omega^c,\\
-\beta_1 \Delta v_1+v_1-\frac{1}{g_1}\nabla \nabla \cdot v_1=0 & \text{\; in \;} \Omega_1 ,\\
-\beta_1 \Delta v_1+v_1-\frac{1}{g_2}\nabla \nabla \cdot v_2=0 & \text{\; in \;} \Omega_2 ,\\
-\beta_2 \Delta v_2+v_2=0 & \text{\; in \;} \Omega^c,\\
-\beta_2 \Delta v_2+v_2-\frac{1}{g_2}\nabla \nabla \cdot v_2=0 & \text{\; in \;} \Omega_2 ,\\
-\beta_2 \Delta v_2+v_2-\frac{1}{g_1}\nabla \nabla \cdot v_1=0 & \text{\; in \;} \Omega_1,
\\
\beta_1 [(\nabla v_1)_{\Omega_1}-(\nabla v_1)_{\Omega^c}]\cdot \vec{\nu}=[p_1^{\ast}-\frac{1}{g_1}(\nabla \cdot v_1)_{\Omega_1})] \vec{\nu}& \text{\; on \;} \Gamma_1,
\\
\beta_2 [(\nabla v_2)_{\Omega_1}-(\nabla v_2)_{\Omega^c}]\cdot \vec{\nu}=[p_1^{\ast}-\frac{1}{g_1}(\nabla \cdot v_1)_{\Omega_1})] \vec{\nu}& \text{\; on \;} \Gamma_1,
\\
\beta_1 [(\nabla v_1)_{\Omega_2}-(\nabla v_1)_{\Omega^c}]\cdot \vec{\mu}=[p_2^{\ast}-\frac{1}{g_2}(\nabla \cdot v_2)_{\Omega_2})]\vec{\mu}& \text{\; on \;} \Gamma_2,
\\
\beta_2 [(\nabla v_2)_{\Omega_2}-(\nabla v_2)_{\Omega^c}]\cdot \vec{\mu}=[p_2^{\ast}-\frac{1}{g_2}(\nabla \cdot v_2)_{\Omega_2})]\vec{\mu}& \text{\; on \;} \Gamma_2,
\\
\beta_1 [(\nabla v_1)_{\Omega_1}-(\nabla v_1)_{\Omega_2}]\cdot \vec{\nu}=[(p_1^{\ast}-p_2^{\ast})+\frac{1}{g_2}(\nabla \cdot v_2)_{\Omega_2})-\frac{1}{g_1}(\nabla \cdot v_1)_{\Omega_1}] \vec{\nu}& \text{\; on \;} \Gamma,
\\
\beta_2 [(\nabla v_2)_{\Omega_1}-(\nabla v_2)_{\Omega_2}]\cdot \vec{\nu}=[(p_1^{\ast}-p_2^{\ast})+\frac{1}{g_2}(\nabla \cdot v_2)_{\Omega_2})-\frac{1}{g_1}(\nabla \cdot v_1)_{\Omega_1}]\vec{\nu}& \text{\; on \;} \Gamma,
\\
(v_1)_{\Omega_1}=(v_1)_{\Omega_2}, \quad (v_2)_{\Omega_1}=(v_2)_{\Omega_2},\quad v_1\cdot \vec{\nu}=v_2\cdot \vec{\nu}, & \text{\; on \;} \Gamma,
\\
(v_1)_{\Omega_1}=(v_1)_{\Omega^c}, \quad (v_2)_{\Omega_1}=(v_2)_{\Omega^c}, & \text{\; on \;} \Gamma_1,
\\
(v_1)_{\Omega_2}=(v_1)_{\Omega^c}, \quad (v_2)_{\Omega_2}=(v_2)_{\Omega^c}, & \text{\; on \;} \Gamma _2,
\\
v_1=v_2=0 & \text{\; on \;} \partial \Theta.
\end{array}
\right.
\end{equation*}
We recall that with our choice of linear growth functions $G_1(s)=g_1(p_1^\ast-s)$, and $G_2(s)=g_2(p_2^\ast-s)$ with $g_1,g_2,p_1^\ast,p_2^\ast>0$, the expression of the pressure $p$ on the interfaces $\Gamma_1, \Gamma_2$ and $\Gamma$ reduces to \eqref{expP}, so that the transmission conditions on the interfaces become,
\begin{equation}
\label{jumpyjumps} \left\{\begin{array}{ll}
\beta_1 [(\nabla v_1)_{\Omega_1}-(\nabla v_1)_{\Omega^c}]\cdot \vec{\nu}=
\beta_2 [(\nabla v_2)_{\Omega_1}-(\nabla v_2)_{\Omega^c}]\cdot \vec{\nu}=[(p)_{\Omega_1}-(p)_{\Omega^c}] \vec{\nu}& \text{\; on \;} \Gamma_1,
\\
\beta_1 [(\nabla v_1)_{\Omega_2}-(\nabla v_1)_{\Omega^c}]\cdot \vec{\mu}
= \beta_2 [(\nabla v_2)_{\Omega_2}-(\nabla v_2)_{\Omega^c}]\cdot \vec{\mu}=[(p)_{\Omega_2}-(p)_{\Omega^c}]\vec{\mu}& \text{\; on \;} \Gamma_2,
\\
\beta_1 [(\nabla v_1)_{\Omega_1}-(\nabla v_1)_{\Omega_2}]\cdot \vec{\nu}
= \beta_2 [(\nabla v_2)_{\Omega_1}-(\nabla v_2)_{\Omega_2}]\cdot \vec{\nu}
=[(p)_{\Omega_1}-(p)_{\Omega_2}]\vec{\nu}& \text{\; on \;} \Gamma.
\end{array}
\right.
\end{equation}

In models incorporating surface tension, the pressure jump depends on the free boundary curvature (see \cite{surf} for example in the case of a Cahn-Hilliard equation). In our case, the presence of the viscosity parameters $\beta_i>0$ is behind the discontinuity of the pressure as can be seen in \eqref{jumpyjumps}. This is in accordance with previous results on Darcy's law \cite{PQF}. In the case where the velocity follows the Brinkman law and taken of gradient form, which is the framework of \citep{PV}, a pressure jump is also shown (see also \citep{velcomp} for explicit computations of the pressure jump in multi-dimensions).

We now prove the existence of a pressure jump. Since, from \eqref{jumpyjumps}, the jumps of the pressure and the jumps of the gradients of the velocities are very related, we first prove that the gradients of the velocities have jumps.

\begin{proposition} Under the assumptions (\ref{cc1})-(\ref{cc3}) and (\ref{cdd1}), the solution $(v_1,v_2) \notin H^2(\Theta)^2,$ with $(v_1,v_2)$ the solution of the system $(T_{VM})$.  \label{v1v2h2}
\end{proposition}

\begin{proof}Let $(v_1,v_2)$ be the solution of the $(T_{VM})$. Then by Theorem \ref{wellpreg} the solution $(v_1,v_2)$ is in $H^1_0(\Theta)^2 \times H^1_0(\Theta)^2 $. Furthermore by Theorem \ref{reg+}, the solution $(v_1,v_2)$ lies in $\mathcal{C}^{\infty}(\overline{\Omega_1})\cap\mathcal{C}^{\infty}( \overline{\Omega_2}) \cap \mathcal{C}^{\infty}( \overline{\Omega^c})$. Then the solution $(v_1,v_2)$ is in $H^2({\Omega_1})\cap H^2( {\Omega_2}) \cap H^2({\Omega^c})$.
We now  proceed by contradiction: let us assume that $v_1,v_2 \in H^{2}(\Theta)^2$ and show that there is a contradiction.

For each of the three first equations in ($T_{VM}$), we multiply (dot product) by $-\Delta v_1$ and integrate on their respective domains. By summing the result, we obtain thanks to an additional Green's identity,
\begin{eqnarray}\label{jumpy1}
\beta_1 \|\Delta v_1\|_{L^2(\Theta)}^2 + \|\nabla v_1\|_{L^2(\Theta)}^2 =- \frac{1}{g_1}\int_{\Omega_1} \nabla \nabla\cdot v_1 \cdot \Delta v_1 - \frac{1}{g_2}\int_{\Omega_2} \nabla \nabla\cdot v_2 \cdot \Delta v_1.
\end{eqnarray}
Using twice Green's identity, we compute the last term as,
\begin{eqnarray}
\nonumber 
&& - \frac{1}{g_2}\int_{\Omega_2} \nabla \nabla\cdot v_2 \cdot \Delta v_1 \\
\nonumber 
&=& \frac{1}{g_2}\int_{\Omega_2} \nabla\cdot v_2 \, \Delta \nabla \cdot v_1 - \frac{1}{g_2} \int_{\partial \Omega_2} \nabla\cdot v_2 \, \Delta v_1 \cdot \vec{\mu} \\
\nonumber 
&=& - \frac{1}{g_2}\int_{\Omega_2} \nabla (\nabla\cdot v_2) \cdot \nabla (\nabla \cdot v_1) + \frac{1}{g_2}\int_{\partial \Omega_2} \nabla\cdot v_2 \, \nabla (\nabla \cdot v_1) \cdot \vec{\mu} - \frac{1}{g_2} \int_{\partial \Omega_2} \nabla\cdot v_2\, \Delta v_1 \cdot \vec{\mu} \, .
\end{eqnarray}
Reinserting in \eqref{jumpy1} and using the same computation for the first term of the right-hand-side, we obtain,
\begin{eqnarray}
\beta_1 \|\Delta v_1\|_{L^2(\Theta)}^2 & + & \|\nabla v_1\|_{L^2(\Theta)}^2 \nonumber \\
= & - &\frac{1}{g_1}\int_{\Omega_1} \nabla (\nabla\cdot v_1) \cdot \nabla (\nabla \cdot v_1) + \frac{1}{g_1}\int_{\partial \Omega_1} \nabla\cdot v_1 \, \left[ \nabla (\nabla \cdot v_1) -  \Delta v_1 \right] \cdot \vec{\nu}
\nonumber \\
& - & \frac{1}{g_2}\int_{\Omega_2} \nabla (\nabla\cdot v_2) \cdot \nabla (\nabla \cdot v_1) + \frac{1}{g_2}\int_{\partial \Omega_2} \nabla\cdot v_2 \, \left[ \nabla (\nabla \cdot v_1) -  \Delta v_1 \right] \cdot \vec{\mu} \,.
\label{jumpy2}
\end{eqnarray}
Now, using the assumption $v_1,v_2 \in H^{2}(\Theta)^2$, we get that $\nabla v_1$ and $\nabla v_2$ are continuous across all interfaces $\Gamma_1$, $\Gamma_2$, $\Gamma$. We therefore deduce from ($T_{VM}$) that,
\begin{eqnarray}
0=[p_1^{\ast}-\frac{1}{g_1}(\nabla \cdot v_1)_{\Omega_1})] \vec{\nu}& \text{\; on \;} \Gamma_1,
\label{nojumpG1} \\
0=[p_2^{\ast}-\frac{1}{g_2}(\nabla \cdot v_2)_{\Omega_2})]\vec{\mu}& \text{\; on \;} \Gamma_2,
\label{nojumpG2} \\
0=[(p_1^{\ast}-p_2^{\ast})+\frac{1}{g_2}(\nabla \cdot v_2)_{\Omega_2})-\frac{1}{g_1}(\nabla \cdot v_1)_{\Omega_1}] \vec{\nu}& \text{\; on \;} \Gamma\, .
\label{nojumpG} 
\end{eqnarray}
Coming back to \eqref{jumpy2}, we get,
\begin{eqnarray} \nonumber
&&\beta_1 \|\Delta v_1\|_{L^2(\Theta)}^2 + \|\nabla v_1\|_{L^2(\Theta)}^2
+ \frac{1}{g_1}\| \nabla (\nabla\cdot v_1) \|_{L^2(\Omega_1)}^2 + \frac{1}{g_2}\int_{\Omega_2} \nabla (\nabla\cdot v_2) \cdot \nabla (\nabla \cdot v_1) \\
\nonumber
&=& \frac{1}{g_1}\int_{\partial \Omega_1} \nabla\cdot v_1 \, \left[ \nabla (\nabla \cdot v_1) -  \Delta v_1 \right] \cdot \vec{\nu}
+ \frac{1}{g_2}\int_{\partial \Omega_2} \nabla\cdot v_2 \, \left[ \nabla (\nabla \cdot v_1) -  \Delta v_1 \right] \cdot \vec{\mu}
\\
\nonumber
&=&
 p_1^\ast \int_{\Gamma_1} \left[ \nabla (\nabla \cdot v_1) -  \Delta v_1 \right] \cdot \vec{\nu}
+ p_2^\ast \int_{\Gamma_2} \left[ \nabla (\nabla \cdot v_1) -  \Delta v_1 \right] \cdot \vec{\mu}
\\ \nonumber &&
+ \int_{\Gamma} \left(\frac{1}{g_1} \nabla\cdot v_1 - \frac{1}{g_2} \nabla\cdot v_2 \right) \, \left[ \nabla (\nabla \cdot v_1) -  \Delta v_1 \right] \cdot \vec{\nu} \\
\nonumber
&=&
 p_1^\ast \int_{\Gamma_1} \left[ \nabla (\nabla \cdot v_1) -  \Delta v_1 \right] \cdot \vec{\nu}
+ p_2^\ast \int_{\Gamma_2} \left[ \nabla (\nabla \cdot v_1) -  \Delta v_1 \right] \cdot \vec{\mu}
+ \int_{\Gamma} \left( p_1^\ast - p_2^\ast \right) \, \left[ \nabla (\nabla \cdot v_1) -  \Delta v_1 \right] \cdot \vec{\nu}
\nonumber \\
&=&
p_1^\ast \int_{\partial \Omega_1} \left[ \nabla (\nabla \cdot v_1) -  \Delta v_1 \right] \cdot \vec{\nu}
+ p_2^\ast \int_{\partial \Omega_2} \left[ \nabla (\nabla \cdot v_1) -  \Delta v_1 \right] \cdot \vec{\mu} \, .
\label{jumpy3}
\end{eqnarray}
Applying a final Green's identity gives, for $j=1,2$,
\begin{equation*}
\int_{\partial \Omega_j} \left[ \nabla (\nabla \cdot v_1) -  \Delta v_1 \right] \cdot \vec{\nu} = \int_{\Omega_j} \nabla \cdot \left[ \nabla (\nabla \cdot v_1) -  \Delta v_1 \right] =0,
\end{equation*}
so that,
\begin{equation*}
\beta_1 \|\Delta v_1\|_{L^2(\Theta)}^2 + \|\nabla v_1\|_{L^2(\Theta)}^2
+ \frac{1}{g_1}\| \nabla (\nabla\cdot v_1) \|_{L^2(\Omega_1)}^2 + \frac{1}{g_2}\int_{\Omega_2} \nabla (\nabla\cdot v_2) \cdot \nabla (\nabla \cdot v_1)
= 0 \, .
\end{equation*}
Following the same procedure for $v_2$, and summing the result, we finally get
\begin{eqnarray*}
\beta_1 \|\Delta v_1\|_{L^2(\Theta)}^2 + \beta_2 \|\Delta v_2\|_{L^2(\Theta)}^2 + \|\nabla v_1\|_{L^2(\Theta)}^2
+ \|\nabla v_2\|_{L^2(\Theta)}^2 & \\
+ \frac{1}{g_1}\| \nabla (\nabla\cdot v_1) \|_{L^2(\Omega_1)}^2
+ \frac{1}{g_2}\| \nabla (\nabla\cdot v_2) \|_{L^2(\Omega_2)}^2 & \\
+ \frac{1}{g_1}\int_{\Omega_1} \nabla (\nabla\cdot v_2) \cdot \nabla (\nabla \cdot v_1)
+ \frac{1}{g_2}\int_{\Omega_2} \nabla (\nabla\cdot v_2) \cdot \nabla (\nabla \cdot v_1) & = 0 \, .
\end{eqnarray*}

Then using Young's inequality on the two last terms we get
\begin{eqnarray}
\beta_1 \|\Delta v_1\|_{L^2(\Theta)^2} + \|\nabla v_1\|_{L^2(\Theta)^2} +\beta_2 \|\Delta v_2\|_{L^2(\Theta)^2} + \|\nabla v_2\|_{L^2(\Theta)^2}
\nonumber \\
+\frac{1}{g_2}\|\nabla \nabla \cdot v_2 \|^2_{L^2(\Omega_2)^2}+\frac{1}{g_1}\|\nabla \nabla \cdot v_1 \|^2_{L^2(\Omega_1)^2}
\nonumber \\
\leq\frac{1}{g_1}\|\nabla \nabla \cdot v_1\|^2_{L^2(\Omega_1)^2}+\frac{1}{4g_1}\|\nabla \nabla \cdot v_2\|^2_{L^2(\Omega_1)^2}+\frac{1}{g_2}\|\nabla \nabla \cdot v_2\|^2_{L^2(\Omega_2)^2}+\frac{1}{4g_2}\|\nabla \nabla \cdot v_1\|^2_{L^2(\Omega_2)^2}
\nonumber 
\end{eqnarray}
We then can write
\begin{eqnarray}
\left(\beta_1-\frac{1}{4g_2}\right) \|D^2 v_1\|^2_{L^2(\Theta)^2} + \|\nabla v_1\|^2_{L^2(\Theta)^2} 
\nonumber 
+\left(\beta_2-\frac{1}{4g_1}\right) \|D^2 v_2\|^2_{L^2(\Theta)^2} + \|\nabla v_2\|^2_{L^2(\Theta)^2} \leq 0.
\end{eqnarray}

The ellipticity condition (\ref{cdd1}) gives positive coefficients in the above inequality. Then we can deduce that
$\nabla v_1\equiv 0$ on $\Theta$.
This is incompatible with \eqref{nojumpG1}, so that we have a contradiction, which proves the result.

\end{proof}

As a corollary we obtain Proposition \ref{p1p2jump}.
\begin{proof}[Proof of Proposition \ref{p1p2jump}]
Taking $\phi \in (\mathcal{C}^{\infty}_c(\Theta))^2$ we have, using the $C^\infty$ regularity of $v_1$ on each subset $\overline \Omega_1$, $\overline \Omega_2$ and  $\overline {\Omega^c}$,
\begin{eqnarray}
\langle\Delta v_1,\phi\rangle &=&\int_{\Theta}v_1 \cdot \Delta \phi\nonumber \\
&=&\int_{\Omega_1} v_1 \cdot \Delta \phi + \int_{\Omega_2} v_1 \cdot \Delta \phi+\int_{\Omega^c} v_1 \cdot \Delta \phi \nonumber \\
&=&\int_{\Omega_1}\Delta v_1 \cdot  \phi +\int_{\Omega_2}\Delta v_1  \cdot \phi +\int_{\Omega^c}\Delta v_1 \cdot  \phi \nonumber \\
&&+\int_{\Gamma_1\cup \Gamma} \nabla \phi \,  v_1\cdot \vec{\nu} +\int_{\Gamma_2\cup \Gamma} \nabla \phi \, v_1\cdot \vec{\mu} +\int_{\Gamma_1\cup \Gamma_2\cup \partial \Theta} \nabla \phi \, v_1\cdot \vec{\tau} \nonumber\\
&& -\int_{\Gamma_1\cup \Gamma}( \nabla v_1)_{\Omega_1}\, \phi \cdot \vec{\nu}  -\int_{\Gamma_2\cup \Gamma} (\nabla v_1)_{\Omega_2} \, \phi \cdot \vec{\mu} -\int_{\Gamma_1\cup \Gamma_2\cup \partial \Theta}(\nabla v_1)_{\Omega^c} \, \phi \cdot \vec{\tau}, \nonumber
\end{eqnarray}
where $\vec{\tau}$ is the outward normal vector to $\Omega^c$. In particular we have $\vec{\tau}=-\vec{\nu}$ on $\Gamma_1$ and $\vec{\tau}=-\vec{\mu}$ on $\Gamma_2$.
Now, let us assume that $p$ is continuous across $\Gamma_1\cup\Gamma_2\cup\Gamma$. 
Then by the transmission conditions in \eqref{jumpyjumps}, we have that the jump of $\nabla v_1\, \vec{\nu}$ across $\Gamma_1\cup\Gamma_2\cup\Gamma$ is zero. Therefore, in the above computation, the integrals on all the interfaces vanish. Finally we can write,
\begin{eqnarray}
\left| \langle\Delta v_1,\phi \rangle \right| \leq \|\phi\|_{L^2(\Theta)} (\|D^2v_1\|_{L^2(\Omega_1)^2}+\|D^2v_1\|_{L^2(\Omega_2)^2})+\|D^2v_1\|_{L^2(\Omega^c)^2}), \nonumber
\end{eqnarray}
then $\Delta v_1 \in L^2(\Theta)^2$ and we obtain $v_1\in H^2(\Theta)^2$. Similarly we get $v_2\in H^2(\Theta)^2$. This is impossible according to Proposition \ref{v1v2h2}, and therefore, we have proved that $p$ cannot be continuous across the interfaces.
\end{proof}

\section{Discussion}\label{sec:discussion}

We presented an evolution model of two proliferating tissues in contact with each other and subject to an enforced segregation (ESVM). Apart from the enforced segregation, the novelty of this model lies in the law of the velocity: it is governed by the Brinkman law considered in a bounded domain, whose boundaries might affect the dynamics of the system.
This setting allows to see swirling motions (non trivial curl) within the tissues, which corresponds to our biological motivation coming from the modelling of the embryo elongation.

We then established the incompressible limit for this model and obtained an incompressible system (L-ESVM) with a geometric description of the free boundaries.
In this step we derived an evolution equation on the repulsion pressure, which, unintuitively, remains present at the limit, even though we attain full segregation. This repulsion pressure affects the dynamics of the shape of the tissues, specifically the velocities of the outer boundaries of the tissues and of their common interface. We called this a \emph{ghost effect}. This effect is also supported by our numerical tests that show a finite effect of the repulsion pressure in the ESVM in its incompressible asymptotic regime.

One first question that arises is: what does this remaining repulsion pressure reveal about the embryo's tissues ?
Having a closer look at our numerical simulations, we observed that the repulsion pressure affects the swirling motions in the PSM. More precisely, it allows adjacent zones of opposite curls to appear in the anterior part of the PSM, a pattern that is observed in the experiments on the bird embryo. Such pattern is not recovered when only passive segregation is assumed, that is, when the repulsion pressure is removed from the model. This suggests that an active segregation is at play to maintain tissue segregation in the vertebrate embryo. This question still requires to be fully explored numerically and compared to experimental data. This will be the subject of a future work.


A second striking feature arising from the model at the incompressible limit is the existence of a pressure jump across the interfaces. We showed this result in particular in situations where the repulsion pressure vanishes.
This feature had already been observed in the specific case where the velocity obeys the Brinkman law and can be written as a gradient in \citep{PV}. This situation is for example met if the system is considered in the whole space. Here, we extend the result to situations in bounded domains with boundary conditions that may affect the dynamics. We mention that the pressure jump should also exist in the case where the repulsion pressure does not vanish, as the jump is viscosity-induced.

The derivation of the incompressible system that we presented here is formal. The case of a velocity that is not in a gradient form, and in the case of a single species, has not been made rigorous yet. We pursue this aim in a future work.

\section{Appendix : Derivation of the transmission problem}\label{sec:appendix}

This section is devoted to the derivation of the transmission problems in $\mathbb{R}^2$. We first do the complete computations of the derivation of the transmission problem in the case of a single species for simplicity. Similar rigorous computations (though not shown in this paper) hold for the transmission problems $(T_2)$ and $(T_{VM})$ respectively for the L-ESVM for a given $q^\infty$ and for the L-VM. We then present a formal version of the transmission problem when $q^{\infty}$ is a stationary solution of (\ref{u1inf})-(\ref{u2inf}) (and not anymore a given $L^2$ function as in the simplified framework of section \ref{sec:studystatLESVM} and section \ref{sec:tran_mod1}).

\subsection{Derivation of the transmission problem for the single species case} 
If we take either densities ($n_1$ or $n_2$) equal to zero in the VM, we obtain a single species model (SS) for the density and the velocity. The system (SS) is as follows: for all $(t,x) \in [0;+\infty)\times \mathbb{R}^{d}$,
\begin{eqnarray}
    && \partial_{t}n+ \nabla \cdot (nv)=nG(p_{\epsilon}), \label{none}\\
    && -\beta\Delta v+v=-\nabla p_{\epsilon},\label{vone} \\
    && p_{\epsilon}=\epsilon \frac{n}{1-n}. \label{pone}
\end{eqnarray}
Notice here that this system, with the general form of the velocity we consider, includes the system in \citep{PV} where the velocity is of gradient form.
In this case, we can also obtain the incompressible limit by taking $\epsilon \xrightarrow{} 0$, and the limiting system can be easily deduced from (\ref{nn1})-(\ref{cong}) by taking either $n_1^\infty$ or $n_2^\infty$ equal to zero, and $q^\infty=0$.

A stationary transmission problem can also be formalized, and we obtain, as a straightforward consequence of the two-species case, the well-posedness and the elliptic regularity results (Theorem \ref{reg+}) as well as the expression of the pressure jump (Proposition \ref{p1p2jump}). Note that the one-species system is elliptic for any $\beta>0$ and $g>0$, so that the results apply without any supplementary condition on these parameters.

The stationary system on the velocity in the single species case is posed in a bounded domain $\Theta \subset \mathbb{R}^2$. We assume $\Omega \subset \Theta$ to be a smooth bounded subdomain with $\overline{\Omega}\subset \Theta$.

The one-species problem coupled with homogeneous Dirichlet boundary conditions on $v^\infty$ is as follows,
$$
(S_1) \;\; 
\left\{
\begin{array}{ll}
-\beta \Delta v^{\infty}+v^{\infty}=- \nabla [(p^{\ast}-\frac{1}{g}\nabla \cdot v^{\infty})\chi_{\Omega}]& \text{\; on \;} \Theta, \label{v'} \\
    v^{\infty}=0& \text{\; on \;} \partial \Theta, 
    \end{array}
     \right.
$$
with $\chi_{\Omega}$ the indicator function of the domain $\Omega$.

\begin{proposition}[Transmission problem, one species]\label{propo:transmission}
Let $\Omega$ be a smooth bounded domain with $\overline{\Omega}\subset \Theta$, and let $\Gamma\coloneqq \partial \Omega$ and $\Omega^c\coloneqq \Theta\backslash\overline{\Omega}$. Let $\beta>0$, $g>0$ and $p^\ast>0$.

Then the solution of system $(S_1)$ solves the following transmission problem $(T_1)$,
 $$
(T_1) \;\; 
\left\{
\begin{array}{ll}
-\beta \Delta v^{\infty} +v^{\infty} -\frac{1}{g} \nabla (\nabla \cdot v^{\infty})= 0 & \text{\; in \;} \Omega,  \\
    -\beta \Delta v^{\infty} +v^{\infty}=0 & \text{\; in \;} \Omega^c, \\
        \beta[(\nabla v^{\infty})_{\Omega}- (\nabla v^{\infty})_{\Omega^c}]\vec{\nu} = [p^{\ast}-\frac{1}{g}(\nabla \cdot v^{\infty})_{\Omega}] \vec{\nu} & \text{\; on \; } \Gamma,\\
        (v^{\infty})_{\Omega}=(v^{\infty})_{\Omega^c} & \text{\; on \; } \Gamma, \\
        v^{\infty}=0 & \text{\; on \; } \partial \Theta,
    \end{array}
     \right.
$$
with $\vec{\nu}$ the outward normal on $\Gamma$ and with the notations $(\cdot)_\Omega$ and $(\cdot)_{\Omega^c}$ defined in \eqref{def_h_D}. 
\end{proposition}

\begin{remark}
Note that in the following proof we only use the $C^{1}$ regularity of $v^\infty$ in each subdomain (up to the boundary) $\overline\Omega$ and $\overline{\Omega^c}$. This ensures that the proof remains valid for the case of more than one species, thanks to Theorem \ref{wellpreg}.
\end{remark}

\begin{proof}
We consider the weak formulation of $(S_1)$,
\begin{eqnarray}
\beta \int_{\Theta} \nabla v^\infty:\nabla \phi + \int_{\Theta} v^\infty\cdot \phi +\int_{\Theta}\left(p^\ast-\frac{1}{g}\nabla \cdot v^\infty\right)\chi_{\Omega}\nabla \cdot \phi=0, \quad \text{ for all } \phi \; \in \mathcal{C}^\infty_c(\Theta)^2, \label{weakform}
\end{eqnarray}
where the Frobenius inner product is used in the first integral.
Then by extension of test functions $\phi \in \mathcal{C}^{\infty}_{c}(\Omega)^2$ (by 0 outside $\Omega$) we obtain,
\begin{equation}
    -\beta \Delta v^{\infty} -\frac{1}{g}\nabla (\nabla \cdot v^{\infty})+v^{\infty}=0 \text{\; \; in  } \mathcal{D}'(\Omega). \nonumber
\end{equation}
Similarly, by extension of test functions $\phi \in \mathcal{C}^{\infty}_{c}(\Omega^c)^2$ (by 0 outside $\Omega^c$) we obtain,
\begin{equation}
    -\beta \Delta v^{\infty} +v^{\infty}=0 \text{\; \; in  } \mathcal{D}'(\Omega^c). \nonumber
\end{equation}

It remains to obtain the transmission conditions at the interface. For this, we use a sequence of smooth functions $(\chi_n)_n$ that are localized around the interface. More precisely these functions satisfy the following properties,
\begin{equation}
\chi_n \in C^\infty(\Theta),
\qquad
\text{for all }n\ge 0 \text{ large enough},
\label{chin_n}
\end{equation}
as well as,
\begin{equation}
    \chi_{n} \xrightarrow[n \to \infty ]{L^{p}(\Theta)}0, \quad  \text{for all } p \in[1,\infty), \label{chin_Lp}
\end{equation}
and finally, for any function $w\in C(\overline{\Omega}) \cap C(\overline{\Omega^c})$,
\begin{equation}
   \lim_{n \to +\infty} \int_\Theta w \cdot \nabla \chi_{n} = \int_{\Gamma} (w)_{\Omega^c} \cdot \vec{\nu} - \int_{\Gamma} (w)_\Omega \cdot \vec{\nu},\label{chin_grad}
\end{equation}
with the notations of \eqref{def_h_D}, with $\vec{\nu}$ the normal vector to $\Gamma$. Such sequence can be constructed explicitely.

Now, we consider a test function $\phi \in \mathcal{C}^{\infty}_{c}(\Theta)^2$, which does not necessarily vanish close to $\Gamma$, and define $\tilde{\phi}_{n}\coloneqq  \phi \chi_n$. We insert the test function $\tilde{\phi}_{n}$  in \eqref{weakform} and get,
\begin{eqnarray}
\beta \int_{\Theta}\nabla v^{\infty} :\nabla \tilde{\phi}_{n} +\int_{\Theta} v^{\infty} \cdot \tilde{\phi}_{n}
     = \int_{\Theta}(p^{\ast}-\frac{1}{g}\nabla \cdot v^{\infty}) \chi_{\Omega}\nabla \cdot \tilde{\phi}_{n},
\end{eqnarray}
which we rewrite,
\begin{eqnarray}
\beta \int_{\Theta} \chi_n \nabla v^{\infty} :\nabla \phi &+& \beta \int_{\Theta}\nabla v^{\infty} :(\nabla \chi_n \otimes \phi) +\int_{\Theta} v^{\infty} \cdot \phi \chi_n\\
     &=& \int_{\Theta} \chi_n (p^{\ast}-\frac{1}{g}\nabla \cdot v^{\infty}) \chi_{\Omega}\nabla \cdot \phi + \int_{\Theta}(p^{\ast}-\frac{1}{g}\nabla \cdot v^{\infty}) \chi_{\Omega}(\nabla \chi_n) \cdot \phi.
\end{eqnarray}
Now using the convergences \eqref{chin_Lp} and \eqref{chin_grad} together with the $C^1$ regularity of $v^\infty$ on each subdomain $\overline{\Omega}$ and $\overline{\Omega^c}$ (up to the boundaries), we can pass to the limit $n\longrightarrow \infty$ and get,
\begin{eqnarray*}
\beta \int_{\Gamma}  \left( \left(\nabla v^{\infty}\right)_{\Omega^c} -\left(\nabla v^{\infty}\right)_{\Omega}\right) :(\vec{\nu} \otimes \phi) 
     =  - \int_{\Gamma}(p^{\ast}-\frac{1}{g}\left(\nabla \cdot v^{\infty}\right)_\Omega) \phi \cdot \vec{\nu}.
\end{eqnarray*}
As this is verified for all $\phi \in \mathcal{C}^{\infty}_{c}(\Theta)^{2}$, particularly by extension of $\phi \in \mathcal{C}^{\infty}_{c}(\Gamma)^{2}$, we have the following condition on the interface,
\begin{equation}
    \beta[(\nabla v^{\infty})_{\Omega}- (\nabla v^{\infty})_{\Omega^c}]\cdot \vec{\nu} = [p^{\ast}-\frac{1}{g}(\nabla \cdot v^{\infty})_{\Omega}]\vec{\nu} \text{\; on \; } \Gamma.
\end{equation}
Then we see that a solution to the variational problem $(S_1)$, using the regularity result and the computations above, is a solution to $(T_1)$,
where the continuity condition comes from the fact that $v^{\infty}$ is in $H^{1}(\Theta)^{2}$ and is then continuous across any hypersurface of $\Theta$.
\end{proof}

This section was dedicated to the rigorous derivation of the transmission problem in the single-species case. Similar computations can be made for the two-species case by adapting the choice of the test function $\tilde{\phi_n}$ to be localized on the interface between the two tissues and on their outer boundaries.

\subsection{A formal transmission problem for the general two species case}
We now present the transmission problem one can obtain formally for the stationary L-ESVM, when $\Theta$, $\Omega_1$, $\Omega_2$ are chosen as in \eqref{cc1}--\eqref{cc3}. Contrarily to the simplified framework of section \ref{sec:studystatLESVM} and section \ref{sec:tran_mod1}, where $q$ is considered given, $q$ is here a stationary solution of (\ref{u1inf})-(\ref{u2inf}). We can rewrite these equations as in \eqref{q1inf} in $\Omega_1$ and \eqref{q2inf} in $\Omega_2$. Then, supposing that $q$ is smooth enough inside each subdomain $\Omega_1$ and $\Omega_2$, the limiting model defined by the equations (\ref{nn1})-(\ref{u2inf}) is equivalent at equilibrium to the following transmission problem $(T_2^*)$ considered on $\Theta $, and coupled with Dirichlet boundary conditions on $\partial \Theta$,

\medskip
\hspace{-0.5cm}
\begin{tabular}{c}
$(T_{2,\Omega_1}^*)
\left\{
  \begin{array}{ll}
  -\beta_1 \Delta v_1+v_1-\frac{1}{g_1}\nabla \nabla \cdot v_1=0 & \text{ in } \Omega_1, \\
  -\beta_2 \Delta v_2+v_2-\frac{1}{g_1}\nabla \nabla \cdot v_1=-\nabla q & \text{ in } \Omega_1,\\
  \nabla \cdot (\log(q+1)v_2)=g_2\log(q+1)[p_2^{\ast}-(p_1^{\ast}-\frac{1}{g_1}\nabla \cdot v_1 +q)] & \text{ in } \Omega_1.
  \end{array}
\right.$\vspace{0.3cm}\\

$(T_{2,\Omega_2}^*)\left\{
  \begin{array}{ll}
  -\beta_1 \Delta v_1+v_1-\frac{1}{g_2}\nabla \nabla \cdot v_2=-\nabla q & \text{ in } \Omega_2, \\
  -\beta_2 \Delta v_2+v_2-\frac{1}{g_2}\nabla \nabla \cdot v_2=0 & \text{ in } \Omega_2, \\
  \nabla \cdot (\log(q+1)v_1)=g_1\log(q+1)[p_1^{\ast}-(p_2^{\ast}-\frac{1}{g_2}\nabla \cdot v_2 +q)] & \text{ in } \Omega_2.
  \end{array}
\right.$
\end{tabular}\vspace{0.3cm}

\hspace{-0.5cm}
\begin{tabular}{c}
$(T_{2,\Omega^c}^*) \left\{
  \begin{array}{ll}
  -\beta_1 \Delta v_1+v_1=0  & \text{ in } \Omega^c,\\
-\beta_2 \Delta v_2+v_2=0  & \text{ in } \Omega^c.\\
  \end{array}
\right.$
\end{tabular}\vspace{0.3cm}

\hspace{-0.5cm}
\begin{tabular}{c}
$(T_{2,\Gamma_1}^*)\left\{
\begin{array}{ll}
\beta_1 [(\nabla v_1)_{\Omega_1}-(\nabla v_1)_{\Omega^c}]\cdot \vec{\nu}=[p_1^{\ast}-\frac{1}{g_1}(\nabla \cdot v_1)_{\Omega_1}] \vec{\nu}  & \text{ on } \Gamma_1,\\
\beta_2 [(\nabla v_2)_{\Omega_1}-(\nabla v_2)_{\Omega^c}]\cdot \vec{\nu}=[p_1^{\ast}+(q)_{\Omega_1}-\frac{1}{g_1}(\nabla \cdot v_1)_{\Omega_1}] \vec{\nu} & \text{ on } \Gamma_1,\\
(v_1)_{\Omega_1}=(v_1)_{\Omega^c}, \quad (v_2)_{\Omega_1}=(v_2)_{\Omega^c}, \quad (v_2q)_{\Omega_1}\cdot \vec{\nu}=0 & \text{ on } \Gamma_1.
\end{array}
\right.$\\
\end{tabular}\vspace{0.3cm}

\hspace{-0.5cm}
\begin{tabular}{c}
$(T_{2,\Gamma_2}^*)\left\{
\begin{array}{ll}
\beta_1 [(\nabla v_1)_{\Omega_2}-(\nabla v_1)_{\Omega^c}]\cdot \vec{\nu}=[p_2^{\ast}+(q)_{\Omega_2}-\frac{1}{g_2}(\nabla \cdot v_2)_{\Omega_2}]\vec{\nu} & \text{ on } \Gamma_2,\\
\beta_2 [(\nabla v_2)_{\Omega_2}-(\nabla v_2)_{\Omega^c}]\cdot \vec{\nu}=[p_2^{\ast}-\frac{1}{g_2}(\nabla \cdot v_2)_{\Omega_2}]\vec{\nu} & \text{ on } \Gamma_2,\\
(v_1)_{\Omega_2}=(v_1)_{\Omega^c}, \quad (v_2)_{\Omega_2}=(v_2)_{\Omega^c}, \quad (v_1q)_{\Omega_2}\cdot \vec{\nu}=0 & \text{ on } \Gamma_2.
\end{array}
\right.$\\
 \end{tabular}\vspace{0.3cm}
 
 \hspace{-0.5cm}
\begin{tabular}{c}
$(T_{2,\Gamma}^*)\left\{
\begin{array}{ll}
\beta_1 [(\nabla v_1)_{\Omega_1}-(\nabla v_1)_{\Omega_2}]\cdot \vec{\nu}=[(p_1^{\ast}-p_2^{\ast})-(q)_{\Omega_2}+\frac{1}{g_2}(\nabla \cdot v_2)_{\Omega_2}-\frac{1}{g_1}(\nabla \cdot v_1)_{\Omega_1}] \vec{\nu} & \text{ on } \Gamma,\\
\beta_2 [(\nabla v_2)_{\Omega_1}-(\nabla v_2)_{\Omega_2}]\cdot \vec{\nu}=[(p_1^{\ast}-p_2^{\ast})+(q)_{\Omega_1}+\frac{1}{g_2}(\nabla \cdot v_2)_{\Omega_2}-\frac{1}{g_1}(\nabla \cdot v_1)_{\Omega_1}]\vec{\nu} & \text{ on } \Gamma,\\
(v_1)_{\Omega_1}=(v_1)_{\Omega_2}, \quad (v_2)_{\Omega_1}=(v_2)_{\Omega_2},\quad v_1\cdot \vec{\nu}=v_2\cdot \vec{\nu}, \quad (q)_{\Omega_1}=(q)_{\Omega_2} & \text{ on } \Gamma.
\end{array}
\right.$
\end{tabular}

\medskip

This transmission problem gives us some insight on the behaviour of the system. It shows on the one hand that $q$ does not have a jump on $\Gamma_1,  \Gamma_2, \Gamma$, unlike the pressure $p$. On the other hand we can observe that the jumps of the gradients of the velocities of the two species at the interface $\Gamma$ are not equal in general.

\section*{Acknowledgement}
PD holds a visiting professor association with the Department of
Mathematics, Imperial College London, UK. 
The authors thank Bertrand Bénazéraf for the stimulating discussions which raised the main questions addressed in this paper. 
MR thanks the CBI for their kind hospitality and support.
\bibliographystyle{plain}
\bibliography{ref}

\end{document}